\documentclass[a4paper,10pt]{article}

\usepackage[english]{babel}
\input xy
\xyoption{all}
\usepackage{comment,xcomment,amssymb,amsmath,amsthm}

\newcommand{\norm}[1]{\left\Vert#1\right\Vert}

\newcommand{\R}{\mathbb{R}}
\newcommand{\im}{\mathrm{Im}\,}         
\newcommand{\lie}[1]{\mathfrak{#1}}     

\newcommand{\h}{\mathbb{H}}

\newcommand{\C}{\mathbb{C}}

\newcommand{\hook}{\lrcorner\,}

\newcommand{\Sp}{\mathrm{Sp}}

\newcommand{\SO}{\mathrm{SO}}

\newcommand{\so}{\mathfrak{so}}

\newcommand{\GL}{\mathrm{GL}}

\newcommand{\dfn}[1]{\emph{#1}}
\newcommand{\id}{\mathrm{Id}}   

\newcommand{\gl}{\lie{gl}}
\newcommand{\Span}[1]{\operatorname{Span}\left\{#1\right\}}

\DeclareMathOperator{\tr}{tr}

\DeclareMathOperator{\End}{End}

\DeclareMathOperator{\Hom}{Hom}

\DeclareMathOperator{\ad}{ad}
\DeclareMathOperator{\Ad}{Ad}
\DeclareMathOperator{\coker}{Coker}

\DeclareMathOperator{\ric}{Ric}

\theoremstyle{plain}
\newtheorem{proposition}{Proposition}
\newtheorem{theorem}[proposition]{Theorem}
\newtheorem{lemma}[proposition]{Lemma}
\newtheorem{corollary}[proposition]{Corollary}

\theoremstyle{definition}
\newtheorem{definition}[proposition]{Definition}
\newtheorem{example}[proposition]{Example}
\theoremstyle{remark}
\newtheorem*{remark}{Remark}

\title{Intrinsic torsion in quaternionic contact geometry}
\author{Diego Conti}

\begin{document}

\maketitle
\begin{abstract}
We investigate quaternionic contact (qc) manifolds from the point of view of intrinsic torsion. We argue that the natural structure group for this geometry is a non-compact Lie group $K$ containing $\Sp(n)\h^*$, and show that  any qc structure gives rise to a canonical $K$-structure with constant intrinsic torsion, except in seven dimensions, when this condition is equivalent to integrability in the sense of Duchemin.

We prove that the choice of a reduction to $\Sp(n)\h^*$ (or equivalently, a complement of the qc distribution) yields a unique $K$-connection satisfying natural conditions on torsion and curvature.

We show that the choice of a compatible metric on the qc distribution determines a canonical reduction to $\Sp(n)\Sp(1)$ and a canonical $\Sp(n)\Sp(1)$-connection whose curvature is almost entirely determined by its torsion. We show that its Ricci tensor, as well as the Ricci tensor of the Biquard connection,  has an interpretation in terms of intrinsic torsion.
\end{abstract}

\renewcommand{\thefootnote}{\fnsymbol{footnote}} 
\footnotetext{\emph{MSC class}: 53C26, 53C10, 53C17}
\renewcommand{\thefootnote}{\arabic{footnote}} 

Quaternionic contact geometry was introduced by Biquard in \cite{Biquard}; its model is the sphere, viewed as the conformal infinity of  quaternionic hyperbolic space. A quaternionic contact (qc) structure is  canonically defined on any $3$-Sasakian manifold, and on more general classes of hypersurfaces in quaternionic manifolds; see \cite{Duchemin:hypersurfaces,IvanovMinchevVassilev}. Explicit examples on Lie groups are also known (see \cite{deAndres:Quaternionic,ContiFernandezSantisteban:qc}). Aside from the link with quaternionic-k\"ahler geometry, a  motivating aspect of qc geometry is the presence of a conformal class of subriemannian metrics for which the Yamabe problem can be studied (see \cite{IvanovMinchevVassilev:Extremals}).

A qc structure on a manifold of dimension $4n+3$ is defined as a distribution of rank $4n$ that can locally  be written as $\ker \eta_1\cap \ker\eta_2\cap\ker\eta_3$, where the $2$-forms $d\eta^s$ define an almost quaternion Hermitian metric on the distribution at each point. Whilst the metric on the distribution is not fixed, its conformal class is determined.  Qc geometry is generally studied by fixing a metric in this class; when $n>1$, this determines a unique Riemannian metric and a metric connection, called  the Biquard connection \cite{Biquard}. For $n=1$,  a similar result holds; however, the resulting connection only has the same features as the Biquard connection when the qc structure is integrable in the sense of Duchemin \cite{Duchemin}.

These connections are defined by torsion conditions which make them unique. The definition of the Riemannian metric involves the choice of a complement to the qc distribution, which is characterized by the existence of a compatible connection with the required torsion conditions. 
One is led to wonder to which extent  the choice of these conditions is canonical.
The literature shows the geometric significance of the curvature of the Biquard connection, and in particular of its Ricci tensor: for instance, it was shown in \cite{IvanovMinchevVassilev,IvanovMinchevVassilev:qcEinstein} that the traceless  Ricci is zero precisely when the qc structure is  $3$-Sasakian up to local homothety, and a Lichnerowicz-type result involving the  Ricci tensor was obtained in \cite{IvanovPetkovVassilev:sharp}. Further, the scalar curvature used in the study of the Yamabe problem is the trace of this Ricci tensor. It is natural to ask  why and whether the Biquard connection and its Ricci tensor are canonical objects of qc geometry.
\smallskip

This paper uses the language of special geometries: namely, of $G$-structures whose intrinsic torsion is partially prescribed.
The intrinsic torsion of a $G$-structure is a tensor representing the first order obstruction to its flatness; it is obtained from the torsion of any connection via a projection to $\coker\partial_G$, where
\[\partial_G\colon (\R^{4n+3})^*\otimes\lie{g}\to\Lambda^2(\R^{4n+3})^*\otimes \R^{4n+3}\]
is induced by the inclusion $\lie{g}\subset  (\R^{4n+3})^*\otimes \R^{4n+3}$. A qc structure cannot be flat in the sense of $G$-structures, for this would make the distribution integrable in the sense of the Frobenius theorem; the type of condition that we will consider is that the  intrinsic torsion take values  in a fixed  $G$-invariant subspace of $\coker\partial_G$.

Understanding what the group $G$ should be is one of the goals of this paper.  As a first step, letting $Q$ be the stabilizer of a point in the Grassmannian of $4n$-planes in $\R^{4n+3}$, we show that  a qc distribution consists in  a $Q$-structure with intrinsic torsion taking values in the orbit $Q\cdot \Theta_0^Q$, where $\Theta_0^Q$ is a distinguished element in $\coker\partial_Q$. The fact that we are dealing with a single $Q$-orbit suggests that a qc $Q$-structure has a canonical reduction, determined by the intrinsic torsion itself. We are led to consider a smaller structure group $B$, namely the stabilizer in $Q$ of $\Theta_0^Q$, obtaining a notion of qc $B$-structure.

The definition gives obvious constraints on the intrinsic torsion of a qc $B$-structure, which we refine using the Bianchi identity and a  calculation involving certain equivariant linear maps. We show that a qc $B$-structure takes values in $B\cdot \Theta_0^B$ when $n>1$; if $n=1$, this is only true up to a $12$-dimensional, irreducible representation of $\SO(4)$, denoted by $S^{5,1}$ in \cite{Duchemin}. In fact, in the course of the paper we prove that this component of the intrinsic torsion is zero if and only if the structure is integrable in the sense of \cite{Duchemin} (an empty condition for $n>1$).

Restricting now to the integrable case, we can repeat the argument and further reduce the structure group to the stabilizer of $\Theta_0^B$, which has the form
\[K=\Sp(n)\h^*\ltimes (\R^{4n})^*,\]
where $(\R^{4n})^*$ is the canonical representation of $\Sp(n)\h^*$ considered in quaternionic geometry, acting as a subgroup of $\Hom(\R^3,\R^{4n})$. Using the Bianchi identity again, we prove that the intrinsic torsion of an integrable qc $K$-structure is constant. In other words, the corresponding $K$-orbit of $\coker\partial_K$ contains a single point, and we cannot repeat the usual procedure. We take this as evidence of the fact that $K$ is the ``natural'' structure group of qc geometry. We emphasize that a $K$-structure does not involve either the choice of a metric on the qc distribution nor of a complement.

The fact that the intrinsic torsion is constant suggests that there could be a canonical connection with constant torsion, obtained by inverting the map
\[\partial_K\colon (\R^{4n+3})^*\otimes\lie{k}\to \im\partial_K.\]
This cannot be done for two reasons:  $\partial_K$ is not injective, so the torsion condition does not define a unique connection, and  secondly $K$ is not reductive, so $\partial_K$ does not even have a $K$-equivariant right inverse. Nonetheless, we are able to show that a canonical connection exists on any $\Sp(n)\h^*$-reduction of a qc $K$-structure; it is characterized by having constant torsion and satisfying natural conditions on the curvature. It will be called the qc connection.

The torsion condition gives strong restrictions on the curvature via the Bianchi identity. A long computation with highest weight vectors allows us to compute the space in which the curvature takes values. This curvature is a stronger invariant than the qc conformal curvature tensor introduced in \cite{IvanovMinchevVassilev}, since it obstructs the existence of a local diffeomorphism with the Heisenberg group that preserves not only the qc distribution, but also the choice of a complement.

In the last part of the paper, we consider integrable qc structures with  a fixed compatible metric on the associated distribution. These can be characterized as $\Sp(n)\Sp(1)$-structures satisfying an intrinsic torsion condition; we refer to them as  quaternionic-contact metric (qcm) structures.

The qc connection has a metric analogue that we call the qcm connection; the two are related via a projection, and  the results on the curvature of the former  carry over to the latter. In fact, the curvature is entirely determined by the intrinsic torsion and its covariant derivative, except for the component $S^4E\subset \lie{sp}(n)\otimes\lie{sp}(n)$, which correponds to the curvature space of hyperk\"ahler manifolds of dimension $4n$.  This leads to a new proof of a result of \cite{IvanovVassilev}, relating closedness of the fundamental four-form to the vanishing of the traceless Ricci tensor for $n>1$.

We show that the intrinsic torsion of a qcm structure consists of three components: one is a  symmetric  tensor which can be identified with the Ricci tensor, one obstructs the integrability of the complement (or ``vertical'' distribution), and one is trivial, determined by the definitions. 

Comparing our results with the literature, we recover the existence of both the Biquard and Duchemin connection on integrable qcm manifolds of arbitrary dimensions. We show that the ``horizontal'' part of the curvature of the Biquard connection is determined linearly by the qcm curvature and torsion; this indicates that formulae using this part of the curvature can  be equally expressed in terms of the qcm curvature. In particular, we prove that the Ricci tensor of the Biquard connection can also be identified with the symmetric part of the intrinsic torsion.

\section{Representations of $\Sp(n)\Sp(1)$}
\label{sec:representations}
The structure group $\Sp(n)\Sp(1)$ plays a central r\^ole in qc geometry; whilst its representation theory is well understood (see  \cite{Swann:Aspects,Salamon:QuaternionicManifolds}), it will be useful to write down some explicit formulae for use in subsequent computations.

Consider the usual inclusion $\Sp(n)\subset\Sp(2n,\C)$
obtained by identifying $\h^n$ with $\C^{2n}$ in such a way that multiplication on the left by $j$ is $\C$-linear. This inclusion induces an identification of $\lie{sp}(2n,\C)$ with the complexification of $\lie{sp}(n)$.

Denoting by $E_{ij}$ the elementary matrix with $1$ at the entry $(i,j)$, the Cartan subalgebra of $\lie{sp}(n)$ is given as
\[\Span{H_1,\dotsc, H_n}, \quad H_k=iE_{kk}\]
and it maps to the standard Cartan subalgebra of $\lie{sp}(2n,\C)$ by
\[H_k \to i(E_{kk}-E_{n+k,n+k}).\]
The weight lattice of the latter is generated by
\[L_k\colon\lie{h}^*\to \C, \quad L_k(H_j)=i\delta_{kj}.\]
We shall denote by $E$ the standard representation 
\[E=\C^{2n}=\Span{v_1,\dotsc, v_{2n}}\]
of $\lie{sp}(2n,\C)$, so that $v_i$ has weight $L_i$ and $v_{n+i}$ has weight $-L_i$. 

The second factor of the product $\Sp(n)\Sp(1)$ has $\lie{sp}(1)$ as its Lie algebra. We shall fix a generator $M$ of its weight lattice with $M(H_1)=-i$, and denote by 
\[H=\C^2=\Span{h_1,h_2}\]
 the standard representation, so that $h_1$ has weight $-M$ and $h_2$ has weight $M$. We can think of $H$ as the $\lie{sp}(1)$-representation given by left multiplication on the quaternions, the identification being given by 
\[a+jb\to ah_1+b h_2, \quad a,b\in\C.\]

We fix the standard ordering for the roots of $\lie{sp}(n)$, and declare $M$ to be positive. This is summarized in Tables~\ref{table:spnroots} and~\ref{table:sp1roots}, which also contain a generator for each root space.
\begin{table}[tb]
\caption{Roots of $\lie{sp}(n)$}
 \label{table:spnroots}
\[\begin{array}{l|l|l|l}
\lie{sp}(2n,\C)&\C\otimes\lie{sp}(n)&\text{root}&>0\\
   \hline
E_{i,n+j}+E_{j,n+i}  & -\frac12(E_{ij}+E_{ji})j+\frac12i\otimes (E_{ij}+E_{ji})k & L_i+L_j,  i\neq j &\text{yes}\\
E_{i,n+i}	&   -\frac12E_{ii}j+\frac12i\otimes E_{ii}k &  2L_i&\text{yes}\\
E_{i,j}-E_{n+j,n+i} & \frac12(E_{ij}-E_{ji})-\frac12i\otimes (E_{ij}+E_{ji})i & L_i-L_j&  i< j\\
E_{n+i,j}+E_{n+j,i}  & \frac12(E_{ij}+E_{ji})j+\frac12i\otimes (E_{ij}+E_{ji})k & -L_i-L_j,  i\neq j & \text{no}\\
E_{n+i,i}	&   \frac12E_{ii}j+\frac12i\otimes E_{ii}k &  -2L_i&\text{no}\\
  \end{array}\]
\end{table}
\begin{table}[tb]
\caption{Roots of $\lie{sp}(1)$}
 \label{table:sp1roots}
\[\begin{array}{l|l|l|l}
\lie{sp}(2,\C)&\C\otimes \lie{sp}(1)&\text{root}&>0\\
   \hline
E_{1,2}	&   -\frac12j+\frac12i\otimes k &  -2M&\text{no}\\
\hline
E_{2,1}	&   \frac12j+\frac12i\otimes k &  2M&\text{yes}\\
  \end{array}\]
\end{table}
The isomorphism $\R^4=\h$ determined by the standard basis $\{1,i,j,k\}$ can be extended to an identification $\R^{4n}=\C^{2n}=\h^n$ via
\[e_{4(j-1)+1}=v_j, \quad e_{4(j-1)+2}=iv_j, \quad e_{4(j-1)+3}=v_{n+j}, \quad e_{4(j-1)+4}=-iv_{n+j}.\]
Here $\{e_1,\dotsc, e_{4n}\}$ is the standard basis of $\R^{4n}$; the dual basis will be denoted by $e^1,\dotsc, e^{4n}$. This induces a representation of $\Sp(n)\h^*$, and hence of its subgroup $\Sp(n)\Sp(1)$,  via
\[(\Sp(n)\times \h^*)\times \h^n\to\h^n, \quad (g,p)\cdot v= gvp^{-1}.\]

This is a real representation whose complexification is well known to be isomorphic to $E\otimes H$, or $EH$ (we shall omit tensor products signs in this context). By choosing a highest weight vector in both representations and applying subsequently negative roots to both sides, one obtains the explicit isomorphism of Table~\ref{table:isowithEH}.
\begin{table}
 \caption{Isomorphism of $\h^n\otimes\C$ with $EH$}
\label{table:isowithEH}
\[\begin{array}{l|l|l}
\R^{4n}\otimes\C & EH & \text{weight}\\
\hline
e_{4(j-1)+1}-ie_{4(j-1)+2}&  v_jh_2& L_j+M\\
e_{4(j-1)j+1}+ie_{4(j-1)+2}&  -v_{n+j}h_1 & -L_j-M\\
e_{4(j-1)+3}-ie_{4(j-1)+4}&  v_j h_1 & L_j-M\\
e_{4(j-1)+3}+ie_{4(j-1)+4}&  v_{n+j} h_2 & -L_j+M
\end{array}\]
\end{table}
Since $\R^{4n}$ is isomorphic to its dual via $e_i\to e^i$, we shall also identify $EH$ with $(\R^{4n})^*$ as an $\Sp(n)\Sp(1)$-module. Thus, we will represent both 
$e_{1}-ie_{2}$ and $e^{1}-ie^{2}$ by $v_1h_2$; this has the consequence that $v_1h_2\hook v_1h_2=0$, but $v_1h_2\hook v_{n+1}h_1=-2$.
\smallskip

We also identify $\Lambda^2(\R^{4n})^*\cong \so(4n)$ by making $\Lambda^2(\R^{4n})^*$ act on $\R^{4n}$ as 
\[\Lambda^2(\R^{4n})^*\times \R^{4n}\to\R^{4n}, \quad (\alpha, v)\to v\hook\alpha.\]
In other words, $e^{i}\wedge e^j$ is identified with $e^i\otimes e_j-e^j\otimes e_i=E_{ji}-E_{ij}$.
Then the action of $\h^*$ on $\R^{4n}$ induces the Lie algebra homomorphism
\begin{equation}
\label{eqn:sp1actionasforms}
\begin{gathered}
Lie(\h^*)=\lie{sp}(1)\oplus\R \to \Span{\id}\oplus\lie{so}(4n), \\
 1\to -\id,\quad i\to -\omega_1,\quad j\to-\omega_2,\quad k\to-\omega_3,
 \end{gathered}
\end{equation}
where (as in \cite{redbook}) the $\omega_s$ satisfy
\begin{gather*}
\omega_1=\frac12 i(v_jh_2\wedge v_{n+j}h_1 + v_jh_1\wedge v_{n+j}h_2),\\
\omega_2+i\omega_3= v_jh_2\wedge v_{n+j}h_2, \quad
\omega_2-i\omega_3=v_jh_1\wedge v_{n+j}h_1.
\end{gather*}
Here and in the sequel, $e^{ij}$ or $e^{i,j}$ stands for the wedge product $e^i\wedge e^j$, and summation over double indices is implied.

It is easy to deduce  that the subspace
\[\Span{\omega_1,\omega_2,\omega_3}\subset \Lambda^2(\R^{4k})^*\]
is fixed under $\h^*$ action; in fact, it is isomorphic to  $\im\h=\lie{sp}(1)$ as an $\h^*$-module via \eqref{eqn:sp1actionasforms}, where $\h^*$ acts on $\im \h$ via
\[\widetilde{\Ad}\colon\h^*\to\End_\R(\im \h), \quad \widetilde{\Ad}(p) q=pq\overline{p}.\]

The tensor product of  two complex representations of $\Sp(n)$ and $\Sp(1)$ has a real structure when both factors have a real or quaternionic structure; all real representations of $\Sp(n)\Sp(1)$ can be written in this way. We denote by $\Lambda^k_0E$ the irreducible representation of $\Sp(n)$ with highest weight $L_1+\dotsc + L_k$, and more generally let $V_{l_1,\dots, l_k}$ be the irreducible representation with highest weight $l_1L_1+\dots + l_kL_k$. It will be understood that $\Lambda^k_0E$ and $V_{l_1,\dots, l_k}$ represent the zero vector space when $n<k$. We have the following
decompositions (see  \cite{Swann:Lancaster}):
\begin{gather*}
 S^2E\otimes E = S^3E\oplus E \oplus V_{21}, \quad \Lambda^2_0E\otimes E = \Lambda^3_0E\oplus E \oplus V_{21},\\
\Lambda^3(EH)=\begin{cases}\Lambda^3_0E S^3H+V_{21}H+E(S^3H+H)& n>1 \\ EH & n=1\end{cases}.
\end{gather*}
We shall also need:
\begin{lemma}
The following isomorphisms of  $\Sp(n)$-modules hold:
\begin{gather*}
\Lambda^2_0E\otimes S^2 E  = V_{31}+V_{211}+\Lambda^2_0E + S^2 E ,\\
S^2E\otimes S^2E = S^4E+V_{31}+V_{22}+S^2E+\Lambda^2_0E+\R.
\end{gather*}
 \end{lemma}
\begin{proof}
The fact that each module on the right hand side appears in the tensor product can be shown by exhibiting a highest weight vector. Moreover, the Weyl Character Formula (see (24.19) in \cite{FultonHarris}) gives
\begin{gather*}
\dim V_{211} = \frac12(n+1)(2n+1)(2n-1)(n-2), \quad 
\dim V_{l,1}=\frac{l(2n-2)}{l+2n-1}\binom{l+2n}{l+1},\\
\dim V_{l,2}=\frac{(l^2+2ln-2n-1)}{2}\binom{l+2n-2}{l+1}.
\end{gather*}
It is now a matter of showing that dimensions add up.
\end{proof}

\section{Distributions as $G$-structures}
\label{sec:gstructures}
In this section we show that qc structures can be characterized in terms of intrinsic torsion using the language of  $G$-structures. The structure group considered in this section, denoted by $Q$, is  inherent in the definition of qc structures, but we will see in later sections that smaller structure groups capture the geometry more completely.

A \dfn{qc structure} on a manifold of dimension $4n+3$ is  a distribution $\mathcal{D}$ of rank $4n$ which can locally be defined as 
\[\mathcal{D}=\ker \eta_1\cap \ker \eta_2\cap \ker \eta_3,\]
where the $\eta_s$ are one-forms such that  $(d\eta_1,d\eta_2,d\eta_3)$ restricted to $\mathcal{D}$ are compatible with an almost quaternion Hermitian metric. The latter condition can be rephrased by requiring the existence   at each point $x$ of a frame
\[u\colon \R^{4n}\to \mathcal{D}_x, \quad d\eta_s(u(e_a),u(e_b))=\omega_s(e_a,e_b),\]
where  the $\omega_s$ are as in Section~\ref{sec:representations}. The qc structure is said to be \dfn{integrable} (in the sense of Duchemin) if in addition at each point there are vectors $R_s$ such that
\[(R_s\hook d\eta^r)|_{\mathcal{D}}+(R_r\hook d\eta^s)|_{\mathcal{D}}=0;\]
this condition turns out to be automatic for $n>1$ (see \cite{Duchemin}). Notice that this is not related to integrability of the distribution $\mathcal{D}$, nor to integrability in the sense of $G$-structures.

Whether qc or not, a codimension three distribution can be viewed as a $Q$-structure, where 
\[Q=\GL(4n,\R)\times\GL(3,\R)\ltimes\Hom(\R^3,\R^{4n})\] 
is the stabilizer of a point in the Grassmannian of $4n$-planes in $\R^{4n+3}$.
We shall denote by $T$ the $Q$-module obtained by letting $Q$ act on $\R^{4n+3}$ via
\[(g,p,h)\colon \begin{pmatrix}v\\ w\end{pmatrix} \to \begin{pmatrix} gv+h(p(w))\\ p(w)\end{pmatrix},\]
denoting by  $e_1,\dotsc, e_{4n+3}$ the standard basis of $T$, and by $e^1,\dotsc, e^{4n+3}$ the dual basis of $T^*$.

We shall write 
\[T=\R^{4n}\oplus\R^3=V\oplus W;\]
notice that $V$ and $W^*$ are $Q$-submodules of $T$ and $T^*$ respectively, but $W$ and $V^*$, despite having a natural $Q$-module structure, are only $\GL(4n,\R)\times\GL(3,\R)$-submodules. On the other hand, $T/V$ is a $Q$-module isomorphic to $W^{**}$.
It will be convenient to denote by $w_s$ the image in $T/V$ of $e_{4n+s}$, and the $e^{4n+s}$ by $w^s$. We shall use the contracted notation $w^{rs}$, $w_{rs}$ for the wedge product of these elements as well.

The intrinsic torsion of a $Q$-structure takes values in a $Q$-module defined as the cokernel of the map
\[\partial_Q\colon T^*\otimes\lie{q}\to \Lambda^2T^*\otimes T\]
obtained by restriction from the map
\[\partial\colon T^*\otimes\lie{gl}(T)\to \Lambda^2T^*\otimes T, \quad e^k\otimes (e^i\otimes e_j)\to e^{ki}\otimes e_j.\]
\begin{lemma}
 The alternating map $\partial_Q$ 
fits into the exact sequence of $Q$\nobreakdash-modules
\[0\to S^2T^*\otimes V+S^2W^*\otimes T\to T^*\otimes\lie{q}\xrightarrow{\partial_Q} \Lambda^2T^*\otimes T\to \Lambda^2V^*\otimes\frac TV\to 0.\]
\end{lemma}
\begin{proof}
Decomposing $T^*\otimes\lie{q}$ into $\GL(4k,\R)\times\GL(3,\R)$-submodules, $\partial_Q$ determines three isomorphisms
\begin{gather*}
 W^*\otimes \lie{gl}(4n,\R) \xrightarrow{\cong}  W^*\otimes V^*\otimes V, \quad
V^*\otimes \lie{gl}(3,\R) \xrightarrow{\cong} W^*\otimes V^*\otimes W,\\
 V^*\otimes \Hom(\R^3,\R^{4n})  \xrightarrow{\cong} W^*\otimes V^*\otimes V, 
\end{gather*}
and three exact sequences
\begin{gather*}
0 \to S^2V^*\otimes V \to V^*\otimes \lie{gl}(4k,\R) \xrightarrow{\partial_Q} \Lambda^2V^*\otimes V\to 0\\
0\to S^2W^*\otimes W\to  W^*\otimes \lie{gl}(3,\R)  \xrightarrow{\partial_Q} \Lambda^2W^*\otimes W\to 0\\
0\to S^2W^*\otimes V\to  W^*\otimes \Hom(\R^3,\R^{4k})   \xrightarrow{\partial_Q} \Lambda^2W^*\otimes V\to 0
\end{gather*}
Thus, the cokernel of  $\partial_Q$ is $\Lambda^2V^*\otimes W$, and the kernel is
\[S^2V^*\otimes V+S^2W^*\otimes T + W^*\otimes V^*\otimes V,\]
which is a $Q$-submodule of $S^2T^*\otimes T$ that can be written as the sum of the two  submodules $S^2T^*\otimes V$ and $S^2W^*\otimes T$ (intersecting non-trivially).
 \end{proof}

This shows that the intrinsic torsion is a map
\[\Theta^Q\colon P\to \Lambda^2V^*\otimes \frac{T}{V}.\]
We can now characterize qc geometry as follows:
\begin{proposition}
\label{prop:qcGstructures}
The distribution associated to a $Q$-structure is integrable if and only if $\Theta^Q$ is identically zero. 
It is qc if and only if $\Theta^Q$ takes values in the $Q$-orbit of 
\[\omega_1\otimes w^{23} + \omega_2\otimes w^{31} + \omega_3\otimes w^{12}.\]
\end{proposition}
\begin{proof}
Let $P$ be a $Q$-structure on $M$. We can represent $\lie{q}$ as a space of block matrices, and decompose connection form, tautological form and torsion as
 \[\omega=\begin{pmatrix} \omega_V & * \\ 0 & \omega_W \end{pmatrix}, \quad \theta=\begin{pmatrix}\theta_V\\\theta_W\end{pmatrix}, \quad \Theta=\begin{pmatrix}\Theta_V\\ \Theta_W\end{pmatrix}.\] 
The $Q$-structure $P$ determines at each point $x\in M$ a projection 
\[h\colon \Lambda^2T^*_xM\to \Lambda^2 (\mathcal{D}_x)^*,\]
where $\mathcal{D}$ is the distribution determined by $Q$. Then, working with a local section $s$ of $P$, we can identify the intrinsic torsion with
\[h(s^*\Theta_W)= h(s^* (d\theta_W+\omega_W\wedge\theta_W))=h(s^*d\theta_W).\]
By the Frobenius theorem,  $\mathcal{D}$ is integrable if and only if  the ideal generated by $s^*\theta_{4n+1}$, $s^*\theta_{4n+2}$ and $s^*\theta_{4n+3}$ is a differential ideal; this is equivalent to 
\[h(s^*(d\theta_W))=0.\]
On the other hand,  $h(s^*\Theta_W)$ is zero if and only if $\Theta^Q$ is zero.

Similarly, for the second part of the statement, the qc condition is equivalent to 
\[h(s^*\Theta_W)=\omega_1\otimes w_{1}+\omega_2\otimes w_{2}+\omega_2\otimes w_{3}\]
for an appropriately chosen section $s$.
\end{proof}

A recurring phenomenon in the study of $G$-structures is that the intrinsic torsion is determined by the exterior derivative of some invariant forms. The structure group $Q$ has no invariant forms; there is, however, an invariant {\em vector-bundle-valued} differential form whose exterior covariant derivative determines the intrinsic torsion. 

This relation is best expressed in  the language of tensorial forms (see e.g. \cite{KobayashiNomizu}). If $P$ is a $Q$-structure and $S$ a $Q$-module, a differential form in $\Omega^k(P,S)$ is called pseudotensorial if it is invariant under the natural action of $Q$ on $\Omega^k(P,S)\cong\Omega^k(P)\otimes S$. It is called tensorial if in addition it is horizontal, i.e. the interior product with any fundamental vector field is identically zero.

Given a connection and a pseudotensorial form $\alpha$ in $\Omega^k(P,S)$, we denote by $D\alpha$ its exterior covariant derivative, as a tensorial $k+1$-form. In particular, if $\theta$ is the tautological form, $\Theta=D\theta$ is the torsion. To  any tensorial form $\alpha$ in $\Omega^k(P,S)$ one can associate an equivariant map 
\[\alpha_\theta\colon P\to \Lambda^kT^*\otimes S, \quad \bigl\langle \alpha_\theta,\frac1{k!}\theta\wedge\dotsm\wedge\theta\bigr\rangle=\alpha,\]
where the angle brackets represent the standard contraction 
\[\Lambda^kT^*\otimes\Lambda^kT\to\R, \quad \langle \eta^1\wedge\dots\wedge\eta^k,X_1\wedge\dots\wedge X_k\rangle = \det(\eta^i(X_j)).\]
With this choice of constants, if $G$ is the trivial group and $\alpha=\theta_1\wedge\dotsb\wedge\theta_k$, then $\alpha_\theta$ is the constant map $\alpha_\theta\equiv e^{1,\dots,k}$.

We shall denote by $\nabla\alpha$ the covariant derivative of $\alpha$, i.e.
\[ \nabla\alpha=D(\alpha_\theta)\in \Omega^1(P,\Lambda^kT^*\otimes S).\]
Given two tensorial forms $\alpha\in\Omega^{h}(P,T)$, $\beta\in\Omega^k(P,S)$, where $\alpha= \alpha^i\otimes e_i$, one can define the interior product $\alpha\hook\beta$ as the tensorial, $S$-valued $h+k-1$-form
\[ (\alpha\hook\beta)_u= \alpha_u^i\wedge (X_i\hook\beta_u), \quad \pi_{*u}(X_i)=u(e_i);\]
for $h=0$, this is the usual interior product. In this notation, if $\alpha$ is a tensorial $k$-form, then
\begin{equation}
 \label{eqn:Dalphanablaalpha}
D\alpha=\bigl\langle \nabla\alpha,\frac1{k!}\theta\wedge\dotsb\wedge\theta\bigr\rangle+\Theta\hook\alpha.
\end{equation}

The structure group $Q$ fixes a tensor
\[w^{123}\otimes w_{123}\in\Lambda^3W^*\otimes\Lambda^3(T/V)\subset\Lambda^3 T^*\otimes\kappa^*,\]
where we have set $\kappa=\Lambda^3W^*$.
Accordingly, on a $Q$-structure $P$ the associated tensorial $3$-form
\[\sigma\in\Omega^{3}(P,\kappa^*), \quad \sigma_\theta\equiv w^{123}\otimes w_{123}\]
is $Q$-invariant, hence parallel.

\begin{proposition}
\label{prop:dsigma22}
Fix a connection on a $Q$-structure $P$. Then the intrinsic torsion $\Theta^Q$
is given by the composition
\[P\xrightarrow{(D\sigma)_\theta} \Lambda^2T^*\wedge\Lambda^2W^*\otimes\kappa^*\xrightarrow{p} \Lambda^2V^*\otimes \Lambda^2W^*\otimes\kappa^*\xrightarrow{c}\Lambda^2V^*\otimes \frac{T}{V}\]
where $p$ is induced by the restriction map $\Lambda^2T^*\to\Lambda^2V^*$ and $c$ is the contraction induced by interior product.
\end{proposition}
\begin{proof}
By \eqref{eqn:Dalphanablaalpha}, since $\sigma$ is parallel,
\[\Theta\hook\sigma= D\sigma;\]
on the other hand $(\Theta\hook\sigma)_\theta$ is obtained from $\Theta_\theta\otimes\sigma_\theta$ via a contraction
\[(\Lambda^2T^*\otimes W)\otimes\Lambda^3W^*\otimes\kappa^*\to \Lambda^2T^*\wedge \Lambda^2W^*\otimes\kappa^*.\]
This shows that $(D\sigma)_\theta$ takes values in $\Lambda^2T^*\wedge\Lambda^2W^*\otimes\kappa^*$, so the composition appearing in the statement is well defined.

It is now straightforward to verify that  $c\circ p\circ (D\sigma)_\theta$ coincides with the projection to $\coker\partial_B$ of the torsion $\Theta$.
\end{proof}

This link between intrinsic torsion and tensorial forms is a recurrent feature of qc geometry; it will be used in Sections~\ref{sec:integrability} and \ref{sec:kintrinsictorsion} to prove vanishing conditions on the intrinsic torsion via the Bianchi identity.

\section{Examples}
\label{sec:example}
In this section we recall three explicit examples of qc structures, which will be used in the rest of the paper for reference. They can be seen as the ``space forms'' of qc geometry, corresponding to the case of positive, negative and zero scalar curvature.
\begin{example}
Consider the sphere as a homogeneous space $G/H$, where \[G=\Sp(n+1)\Sp(1), \quad H=\Sp(n)\Sp(1).\] The Lie algebra of $G$ is
\[\lie{g}=\biggl\{\biggl(\begin{pmatrix} a & b \\- \overline{b}^T & d\end{pmatrix} , q\biggr) \mid a\in\lie{sp}(n), b\in\h^n, d,q\in\im \h\biggr\},\]
and $\lie{h}$ is defined by $b=0$, $d=q$.

We set
\begin{gather*}
w_1=(iE_{nn},-i)\quad w_2=(jE_{nn},-j),\quad w_3=(kE_{nn},-k)\\
e_{4l+1}=-E_{n,l+1}+E_{l+1,n}, \quad e_{4l+2}=i(E_{n,l+1}+E_{l+1,n}), \quad e_{4l+3}=j(E_{n,l+1}+E_{l+1,n}), \\ e_{4l+4}=k(E_{n,l+1}+E_{l+1,n}),\quad 0\leq l\leq n-1.
 \end{gather*}
Then $e_1,\dotsc, e_{4n},w_1,w_2,w_3$ define a frame on a complement $\lie{m}$ of $\lie{h}$ in $\lie{g}$, hence an $\Sp(n)\Sp(1)$-structure; this is invariant under $H$, so it defines a global structure on $G/H$. 
We compute
\[de^a |_{\Lambda^2\lie{m}} =  (e_a\hook \omega_s)\wedge w^s, \quad dw^s|_{\Lambda^2\lie{m}} =\omega_s.\]
The projection to $\lie{h}$ defines a connection; its torsion is 
\[\Theta= (e_a\hook \omega_s)\wedge w^s\otimes e_a + \omega_s\otimes w_s.\]
This shows immediately that $\Theta^Q=\omega_s\otimes w_s$, so by Proposition~\ref{prop:qcGstructures} this is a qc structure. Consistently with Proposition~\ref{prop:dsigma22},
\[D\sigma = (\omega_1\wedge w^{23}+\omega_2\wedge w^{31}+\omega_3\wedge w^{12})\otimes w_{123}.\]
We note for future reference that the curvature of this connection is 
 \[\Omega=-\sum_{a<b} e^{ab}\otimes e^a\wedge e^b - \sum_{a<b,s} e^{ab}\otimes e_a\hook\omega_s \wedge e_b\hook\omega_s -(\omega_s-2w_s\hook w^{123})\otimes \omega_s.\]
\end{example}

\begin{example}
A similar example is the homogeneous space $G/H$, where
\[G=\Sp(n,1)\Sp(1), \quad H=\Sp(n)\Sp(1).\]
In this case we choose a complement $\lie{m}$ spanned by
\begin{gather*}
w_1=(-iE_{nn},i),\quad w_2=(-jE_{nn},j),\quad w_3=(-kE_{nn},k),\\
e_{4l+1}=E_{n,l+1}+E_{l+1,n}, \quad e_{4l+2}=i(-E_{n,l+1}+E_{l+1,n}), \\ e_{4l+3}=j(-E_{n,l+1}+E_{l+1,n}), \quad e_{4l+4}=k(-E_{n,l+1}+E_{l+1,n}).
 \end{gather*}
The connection defined by the projection has  torsion 
\[\Theta= -(e_a\hook \omega_s)\wedge w^s\otimes e_a + \omega_s\otimes w_s\]
and curvature 
 \[\Omega=\sum_{a<b} e^{ab}\otimes e^a\wedge e^b + \sum_{a<b,s} e^{ab}\otimes e_a\hook\omega_s \wedge e_b\hook\omega_s +(\omega_s+2w_s\hook w^{123})\otimes \omega_s.\]
\end{example}

\begin{example}
The remaining example to consider is the Heisenberg group, which is characterized by the existence of a left-invariant basis of one-forms $e^1,\dotsc, e^{4n+3}$ with
\[de^a=0, \quad de^{4n+s}=\omega_s,\]
where the $\omega_s$ are defined in terms of the $e^a$ in the usual way. In this case we can use the $(-)$ connection, i.e. the connection for which the indicated frame is parallel. It has torsion $\Theta_0$ and curvature zero.
\end{example}

A common feature of these examples is the presence of a natural connection, which will play a r\^ole in Section~\ref{sec:posit}. Other homogeneous examples appear in  \cite{deAndres:Quaternionic, ContiFernandezSantisteban:qc}. 
Hypersurfaces in quaternionic manifolds also give rise to qc structures under certain conditions, as shown in  \cite{Duchemin:hypersurfaces,IvanovMinchevVassilev}.

\section{Qc dialectics}
In the language of Proposition~\ref{prop:qcGstructures}, a qc structure is  a $Q$-structure whose 
intrinsic torsion  is in the $Q$-orbit of $\Theta_0^Q$, which we define as the image in  $\Lambda^2V^*\otimes \frac{T}{V}$ of
\[\Theta_0=\omega_1\otimes e_{4n+1} + \omega_2\otimes e_{4n+2} + \omega_3\otimes e_{4n+3}\in \Lambda^2T^*\otimes T.\]
This means that any qc structure has a natural reduction to $B$, where $B$ denotes the stabilizer in $Q$ of $\Theta_0^Q$. In this section we determine $B$ and study  the space $\coker\partial_B$ in which the intrinsic torsion of $B$-structures takes values.

The forms $\omega_s$ are fixed by the action of $\Sp(n)$. More generally, we say three elements  $\gamma_1$, $\gamma_2$, $\gamma_3$ of $\Lambda^2\R^{4n}$ are compatible with an $\Sp(n)$-structure if 
 some linear isomorphism of $\R^{4n}$  maps each $\gamma_s$ in $\omega_s$.
\begin{lemma}
\label{lemma:stabilizertriplet}
Let $\Gamma\subset\Lambda^2\R^{4n}$ be the space spanned by three $2$-forms compatible with an $\Sp(n)$-structure. Then the structure is uniquely determined by $\Gamma$, up to $\h^*$ action.
\end{lemma}
\begin{proof}
 Let $\gamma_1$, $\gamma_2$, $\gamma_3$ be compatible with an $\Sp(n)$-structure. Each two-form $\gamma_s$ defines an isomorphism
\[\gamma_s\colon \R^{4k}\to (\R^{4k})^*, \quad X\to X\hook\omega_s.\]
Three complex structures are induced on $\R^{4n}$ by
\[J_3=-\gamma_2^{-1}\gamma_1 =  \gamma_1^{-1}\gamma_2\]
and cyclic permutations.
It follows that 
\begin{equation*}
 (a\gamma_1+b\gamma_2+c\gamma_3)^{-1}=\frac1{a^2+b^2+c^2}a\gamma_1^{-1}+b\gamma_2^{-1}+c\gamma_3^{-1}.
\end{equation*}
We can therefore define  a linear map on the space $\Gamma$ spanned by the $\gamma_s$,
\[\Gamma\to \Hom((\R^{4n})^*,\R^{4n}), \quad \gamma\to \gamma^\dagger = \norm{\gamma}^2\gamma^{-1}, \quad \gamma\neq0.\]
By construction
\[\alpha^\dagger\beta + \beta^\dagger\alpha=\langle\alpha,\beta\rangle\id, \quad \alpha,\beta\in\Gamma.\]

Take three elements $\alpha,\beta,\gamma$ in $\Gamma$ and assume they are also compatible with an $\Sp(n)$-structure. Then
\[\alpha^{-1}\beta+\beta^{-1}\alpha=\norm{\alpha}^{-2} \alpha^\dagger\beta +\norm{\beta}^{-2} \beta^\dagger\alpha=0,\]
leading to
\[\norm{\alpha}^{-2} \alpha^\dagger\beta -\norm{\beta}^{-2} \alpha^\dagger\beta +\norm{\beta}^{-2}\langle \alpha,\beta\rangle\id=0.\]
Now observe that $\alpha^\dagger\beta$ is not a multiple of the identity whenever $\alpha,\beta$ are linearly independent.
Thus, $\alpha,\beta$ are orthogonal with the same norm. 

Summing up, $\alpha$, $\beta$ and $\gamma$ form an orthogonal basis of $\Gamma$ of elements with the same norm; this basis is positevely oriented by construction, and so uniquely determined up to $\h^*$ action.
\end{proof}

Recall from Section~\ref{sec:representations} that both $\h^n$ and $\im\h$ (and therefore its dual $(\im\h)^*$) are equipped with a left $\h^*$-action. The identification $\R^{4n}=\h^n$ induces 
\[\rho\colon \GL(n,\h)\times\h^*\to\GL(4n,\R), \quad \rho(g,p)(v)=gvp^{-1}.\]
\begin{proposition}
The stabilizer in $Q$ of $\Theta_0^Q$ is the group
\[B=\Sp(n)\h^*\ltimes\Hom(W,V),\] 
where the first factor represents the image of the homomorphism
 \[\iota\colon\Sp(n)\times\h^*\to \GL(V)\times\GL(W),\quad 
(g,p) \to (\rho(g,p),\widetilde\Ad(p))\]
having implicitly identified $W$ with $(\im\h)^*$.
\end{proposition}
\begin{proof}
It is clear that $\Sp(n)\ltimes\Hom(W,V)$ fixes $\Theta_0^Q$. As for $\h^*$, observe that $\iota$ makes
\[\Span{\omega_1,\omega_2,\omega_3}\subset\Lambda^2V^*\] isomorphic to $\im \h$ as a representation of $\h^*$ (see Section~\ref{sec:representations}). Thus, identifying $W$ with $(\im\h)^*$, $\Theta_0^Q$ is in the trivial submodule of 
$\im \h\otimes (\im\h)^*$.

Conversely, we must show that the stabilizer of $\Theta_0^Q$ in $\GL(V)\times\GL(W)$ is $\iota(\Sp(n)\times\h^*)$. In fact, the stabilizer of 
$\Span{\omega_1,\omega_2,\omega_3}$
in $\GL(V)\times\GL(W)$ is $\Sp(n)\h^*\times\GL(W)$ by Lemma~\ref{lemma:stabilizertriplet}. On the other hand $\Theta_0^Q$ determines an isomorphism
\[W^*\to \Span{\omega_1,\omega_2,\omega_3}\]
and the subgroup of $\Sp(n)\h^*\times\GL(W)$ that fixes this isomorphism is precisely the image of $\iota$.
\end{proof}

\begin{remark}
The identification of $W$ with $(\im\h)^*$ has the consequence that scalars $\lambda\in\R^*\subset B$ act on $T$ as $\lambda^{-2}\id_W+\lambda^{-1}\id_V$.
\end{remark}

We can now refine the second part of Proposition~\ref{prop:qcGstructures} in the following way:
\begin{corollary}
Every qc $Q$-structure has a unique $B$-reduction $P$
such that \[\Theta^Q(u)=\Theta_0, \quad u\in P.\]
\end{corollary}
We shall refer to such a structure as a \dfn{qc $B$-structure}. This leads us to consider the intrinsic torsion of $B$-structures. Consider the diagram
\[\xymatrix{T^*\otimes\lie{b}\ar[d]\ar[r]^{\partial_B}& \Lambda^2T^*\otimes T\ar[d]\ar[r] & \coker\partial_B\ar[r]\ar[d]^r&0\\ T^*\otimes\lie{q}\ar[r]^{\partial_Q}&\Lambda^2T^*\otimes T\ar[r] &  \Lambda^2V^*\otimes \frac{T}V \ar[r] & 0}\]
By construction $r$ is $B$-equivariant, and maps the intrinsic torsion of a $B$-structure to its $Q$-intrinsic torsion, i.e. the intrinsic torsion of the induced $Q$-structure.  Accordingly, a qc $B$-structure has intrinsic torsion in $r^{-1}(\Theta_0^Q)$. In fact, we will see in Section~\ref{sec:integrability}  that the intrinsic torsion of a qc $B$-structure is forced to lie in a much smaller space. For the moment, we use the above diagram to study $\coker\partial_B$.

\begin{lemma}
\label{lemma:projectiontorsion}
The projection $ \Lambda^2T^*\otimes T\to\coker\partial_B$ induces by restriction $B$-equivariant maps
\[p_{VVV}\colon  \Lambda^2V^*\otimes V \to \coker\partial_B \quad
p_{VWW}\colon V^*\otimes W^*\otimes W \to\coker\partial_B
\]
such that
\[\ker r=\im p_{VVV}+\im p_{VWW} \quad\text{(not a direct sum)}.\]
\end{lemma}
\begin{proof}
We first prove that
\begin{equation}
\label{eqn:coker}
\im\partial_B\supset \Lambda^2W^*\otimes T+V^*\otimes W^*\otimes V.
\end{equation}
Indeed, the maps
\[\partial_B\colon V^*\otimes \Hom(W,V)  \to \Lambda^{1,1}\otimes V,\quad 
\partial_B\colon W^*\otimes \Hom(W,V)  \to \Lambda^{0,2}\otimes V\]
are obviously surjective; moreover, the composition map
\[W\otimes \lie{sp}(1) \xrightarrow{\partial_B}  \Lambda^2W^*\otimes W+(V\otimes W)\otimes V \to \Lambda^2W^*\otimes W\]
is an isomorphism. This proves  \eqref{eqn:coker}.

The projection $\Lambda^2T^*\otimes T\to\coker\partial_B$
induces a $B$-equivariant map
\[p\colon\Lambda^2T^*\otimes V\to\coker\partial_B;\]
by \eqref{eqn:coker}, $p$ factors through
\[\Lambda^2T^*\otimes V\to\Lambda^2V^*\otimes V\to \coker\partial_B.\]
Composing on the left with the inclusion $\Lambda^2V^*\otimes V\to \Lambda^2T^*\otimes V$ (which is not $B$-equivariant) gives an equivariant map.

Similarly, the map
\[T^*\wedge W^*\otimes T\to\coker\partial_B\]
factors through
\[T^*\wedge W^*\otimes T\to V^*\otimes W^*\otimes W\to \coker\partial_B.\]

Finally, \eqref{eqn:coker} implies that
the map
\[\Lambda^2V^*\otimes T+ V^*\otimes W^*\otimes W \to \coker\partial_B\]
is surjective, yielding the final part of the statement.
\end{proof}
\begin{remark}
It is not possible to construct an analogous $B$-equivariant map 
\[p_{VVW}\colon \Lambda^2V^*\otimes W\to\coker\partial_B.\]
Indeed, the smallest $B$-module in $\Lambda^2T^*\otimes T$ containing $\Lambda^2V^*\otimes W$ is $\Lambda^2V^*\otimes T$, and the image of $\Lambda^2V^*\otimes T$ in the cokernel is bigger than the image of $\Lambda^2V^*\otimes W$.
\end{remark}

It follows from the above remark that we cannot think of $\Theta_0^Q$ as a ``component'' of the $B$-intrinsic torsion: we have to express the relation in terms of a short exact sequence.
\begin{proposition}
\label{prop:kernelr}
There is an exact sequence of $B$-modules
\begin{equation}
\label{eqn:nonsplittingseq}
0\to W_1\oplus W_2\xrightarrow{i} \coker\partial_B\xrightarrow{r}\coker\partial_Q\to 0
 \end{equation}
where 
\begin{align*}
W_1&=V^*\otimes \lie{sp}(n)^\perp\cong \begin{cases}(V_{21}+\Lambda^3_0E+2E)(S^3H+H), &  n>1,\\ ES^3H+EH, & n=1,\end{cases}\\
W_2&=V^*\otimes S^2_0(W)\cong ES^3H+ES^5H,
\end{align*}
and the restriction of $i$ to each component is given by restricting  the $B$-equivariant alternating maps
\begin{gather*}
\partial_1\colon V^*\otimes \gl(V)\to \Lambda^2V^*\otimes V,\quad
\partial_2\colon V^*\otimes \gl(W)\to V^*\otimes W^*\otimes W.
\end{gather*}
\end{proposition}
In this statement, $\lie{sp}(n)^\perp$ denotes the orthogonal complement of $\lie{sp}(n)$ in $\so(4n)$.

\begin{proof}
As an  $\Sp(n)\h^*$-module, the Lie algebra of $B$ decomposes as
\[\lie{b}=\lie{sp}(n)\oplus\lie{sp}(1)\oplus\R\oplus\Hom(W,V),\]
where the inclusion of $\lie{sp}(1)\oplus\R$ in $\gl(T)$ is given by
\[p \to (R_{-p},\widetilde\ad(p)),\quad  p\in\lie{sp}(1); \quad \lambda\to (-\lambda\id, -2\lambda\id), \quad \lambda\in\R.\]
By Lemma~\ref{lemma:projectiontorsion}
\[\ker r = \frac{\Lambda^2V^*\otimes V + V^*\otimes W^*\otimes W}{\im \partial_B\cap (\Lambda^2V^*\otimes V + V^*\otimes W^*\otimes W)} =\frac{\Lambda^2V^*\otimes V + V^*\otimes W^*\otimes W}{\partial_B(V^*\otimes \lie{sp}(n)\h)};\]
using the fact that $\partial_2$ is injective, the snake lemma applied to 
\[\xymatrix{
0\ar[r]&V^*\otimes\lie{sp}(n)\ar[r]\ar[d]^{\partial_{1}} & V^*\otimes\lie{sp}(n)\h\ar[d]^{\partial}\ar[r]& V^*\otimes\h\ar[r]\ar[d]^{\partial_2}&0\\
0\ar[r]& \Lambda^2V^*\otimes V\ar[r]&\Lambda^2V^*\otimes V+V^*\otimes W^*\otimes W\ar[r]& V^*\otimes W^*\otimes W\ar[r] &0
 }\]
yields
\[0\to W_1\to \ker r\to W_2\to 0.\]
This sequence splits by Lemma~\ref{lemma:projectiontorsion}. 
Now observe that $\Hom(W,V)$ acts trivially on $W_1$ and $W_2$, and the component $\R$ acts as a multiple of the identity on $\Lambda^2V^*\otimes V + V^*\otimes W^*\otimes W$. The decomposition of $W_1$ and $W_2$ into irreducibile $B$-modules is therefore the same as the decomposition into $\Sp(n)\Sp(1)$-modules.

\end{proof}
  
\begin{remark}
\label{remark:snake}
An alternative description can be obtained by applying the snake lemma to 
\[\xymatrix{
0\ar[r]&T^*\otimes\lie{b}\ar[r]\ar[d]^{\partial_{\lie{b}}} & T^*\otimes\lie{q}\ar[d]^{\partial_{\lie{q}}}\ar[r]& T^*\otimes\frac{\lie{q}}{\lie{b}}\ar[r]\ar[d]&0\\
0\ar[r]& \Lambda^2T^*\otimes T\ar[r]&\Lambda^2T^*\otimes T\ar[r]& 0
 }\]
giving an exact sequence
\[\ker\partial_Q\xrightarrow{\alpha} T^*\otimes\frac{\lie{q}}{\lie{b}}\to \coker\partial_B\to\coker\partial_Q\to 0\]
This means that 
\[\coker\alpha=W_1+W_2,\]
i.e. we can think of the $W_i$ as the components of $\coker\alpha$.
\end{remark}

\begin{corollary}
\label{cor:kerpartialB}
The kernel\ of $\partial_B$ is $S^2W^*\otimes V+ (S^2ES^2H+S^2H)$, where the second summand  lies diagonally in $W^*\otimes(\lie{sp}(n)+\h)+V^*\otimes\Hom(W,V)$.
\end{corollary}
\begin{proof}
By Proposition~\ref{prop:kernelr}, the kernel of $\partial_B$ is an $\Sp(n)\Sp(1)$-module of the same dimension as $S^2W^*\otimes V+ S^2H(S^2E\oplus\R)$, and clearly it contains $S^2W\otimes V$. The restriction of $\partial$ to $V^*\otimes \Hom(W,V)$ is injective with image $V^*\otimes W^*\otimes V$, so it contains  $\partial_1(W^*\otimes(\lie{sp}(n)+\h))$. Taking the kernel of $\partial_2$ in $W^*\otimes (\lie{sp}(n)+\h)$, we find  $S^2W^*\otimes V+ (S^2ES^2H+S^2H)$.
\end{proof}

In later sections we shall have to work with certain invariant maps. Since the decomposition of $W_1+W_2$ into $\Sp(n)\Sp(1)$-modules contains some modules with multiplicity greater than one, it is clear that Schur's lemma will not be sufficient in order to study  these maps. Thus, we shall have to be more explicit. Since we refer to $\Sp(n)\Sp(1)$-modules, we can identify $T$, $V$ and $W$  with their duals through the metric here.

Assume first $n>1$. The space $V\otimes \Lambda^2V$ contains three copies of $EH$, corresponding to the highest weight vectors
\begin{align*}
\alpha_1
 &=  v_1h_2\otimes (v_{n+j}h_1\wedge v_jh_2+v_{n+j}h_2\wedge v_jh_1)
 +2 v_1h_1\otimes  v_jh_2\wedge v_{n+j}h_2\\
\alpha_2&=v_{n+j}h_2\otimes (v_jh_2\wedge v_1h_1+v_1h_2\wedge v_jh_1)-v_jh_2\otimes(v_{n+j}h_2\wedge v_1h_1+v_1h_2\wedge v_{n+j}h_1)\\
\alpha_3&= v_{n+j}h_2\otimes (v_1h_1\wedge v_jh_2 + v_1h_2\wedge v_jh_1)-v_{j}h_2\otimes (v_1h_1\wedge v_{n+j}h_2+v_1h_2\wedge v_{n+j}h_1)\\
& -2v_{n+j}h_1\otimes (v_1h_2\wedge v_jh_2) +2v_{j}h_1\otimes (v_1h_2\wedge v_{n+j}h_2).
\end{align*}
Similarly, $V\otimes \Lambda^2V$  contains the linearly independent highest weight vectors 
\[
\beta_1 
 =  v_1h_2\otimes ( v_jh_2\wedge v_{n+j}h_2), \quad
\beta_2= v_jh_2\otimes (v_1h_2\wedge v_{n+j}h_2)-v_{n+j}h_2 \otimes  (v_1h_2\wedge v_jh_2),
\]
each generating a submodule isomorphic to  $ES^3H$.

We shall denote by $\tilde\beta_i$, $\tilde\alpha_i$ the images of these vectors under the isomorphism
\begin{equation*}
\widetilde{\cdot}\colon V\otimes \Lambda^2V\to \Lambda^2V\otimes V, \quad \widetilde{v\otimes\eta}= \eta\otimes v.
\end{equation*}
Each of $V\otimes S^2_0W$ and $V\otimes \Lambda^2W$ contains an $ES^3H$, with highest weight vectors
\begin{gather*}
\beta_3 =v_1h_2\otimes (w_1\otimes (w_2+iw_3)+ (w_2+iw_3)\otimes w_1)-2iv_1h_1\otimes (w_2+iw_3)\otimes (w_2+iw_3),\\
\beta_4=v_1h_2\otimes w_1\otimes (w_2+iw_3)-v_1h_2 \otimes (w_2+iw_3)\otimes w_1.
 \end{gather*}
Finally, $V\otimes W\otimes W$  contains two copies of $EH$, generated by
\begin{gather*}
\alpha_4=v_1h_2\otimes (w^1\otimes w_1+w^2\otimes w_2+w^3\otimes w_3),\\
\alpha_5=v_1h_2\otimes (w^2\otimes w_3-w^3\otimes w_2) +v_1h_1\otimes (w^1\otimes (w_2+iw_3)-(w^2+iw^3)\otimes w_1).
\end{gather*}

If $n=1$, all these highest weight vectors remain well defined, although they do not generate distinct modules, as $\alpha_3=\alpha_1$ and $\beta_2=\beta_1$. Accordingly, we can drop the assumption on $n$ for the rest of the section.
\begin{lemma}
\label{lemma:formulae}
We have
\begin{align*}
\partial(\alpha_1) &=\frac12\tilde\alpha_3-\frac32\tilde\alpha_2, &
 \partial(\alpha_2) &=-\tilde\alpha_1-\frac12\tilde\alpha_3+\frac12\tilde\alpha_2,\\
 \partial(\alpha_3)&=\tilde\alpha_1-\frac12\tilde\alpha_3-\frac32\tilde\alpha_2,\\
\partial \beta_1&=
-\tilde\beta_2, &
\partial\beta_2
&=\tilde\beta_2-2\tilde\beta_1,
\end{align*}
and $\partial(V^*\otimes (\lie{sp}(n)+\h))\subset\im \partial_B$ contains
\[ \partial(\alpha_2), \quad \tilde\alpha_2+\tilde\alpha_3+8\alpha_4, \quad 8i\alpha_5-\tilde\alpha_3+3\tilde\alpha_2, \quad
\tilde\beta_2+2i\beta_4.\]
\end{lemma}
 \begin{proof}
The first part is a straightforward computation.
For the second part, observe that by Section~\ref{sec:representations} $\omega_2+i\omega_3\in\lie{sp}(1)\otimes\C$ acts on $W\cong\im\h$ as
\[2(-w^3\otimes w_1+w^1\otimes w_3-iw^1\otimes w_2+iw^2\otimes w_1)=-2i(w^1\otimes (w_2+iw_3)-(w^2+iw^3)\otimes w_1),\]
and $\omega_1$ acts as
\[-2w^2\otimes w_3+2w^3\otimes w_2.\]
It follows that $4i\alpha_5-\alpha_1$,   $\beta_1-2i\beta_4$ lie in $V\otimes\lie{b}$.

Now recall that $-\id_V-2\id_W$ lies in $\lie{b}$, and
\[\id_V=\frac12\left(-v_jh_2\otimes v_{n+j}h_1-v_{n+j}h_1\otimes v_jh_2+v_jh_1\otimes v_{n+j}h_2+v_{n+j}h_2\otimes v_jh_1\right),\]
giving 
\[\im\partial_B\ni\partial(v_1h_2\otimes (-2\id_V-4\id_W))=-\frac12\tilde\alpha_2-\frac12\tilde\alpha_3-4\alpha_4.\qedhere\]
\end{proof}
\begin{proposition}
\label{prop:basis}
The components isomorphic to $EH$ and $ES^3H$ inside $W_1$,  $\partial_1(W_1)$ and $W_2$ are identified by 
\[
W_1\ni \beta_1,\beta_2, \alpha_1, \alpha_3, \quad \partial_1(W_1)\ni \tilde\beta_1,\tilde\beta_2, \tilde\alpha_1-\tilde\alpha_3,\tilde\alpha_1-3\tilde\alpha_2, \quad  W_2\ni \beta_3;\\ 
\]
moreover  the following equivalences modulo $\im\partial_B$ hold:
\begin{gather*}
\begin{aligned}
 \tilde\alpha_1 &\equiv \frac38(\tilde\alpha_1-\tilde\alpha_3)- \frac18(\tilde\alpha_1-3\tilde\alpha_2), &
 \tilde\alpha_2 &\equiv \frac18(\tilde\alpha_1-\tilde\alpha_3)- \frac38(\tilde\alpha_1-3\tilde\alpha_2), \\
 \tilde\alpha_3 &\equiv -\frac58(\tilde\alpha_1-\tilde\alpha_3)- \frac18(\tilde\alpha_1-3\tilde\alpha_2), &
 \alpha_4 &\equiv \frac1{16}(\tilde\alpha_1-\tilde\alpha_3)+ \frac1{16}(\tilde\alpha_1-3\tilde\alpha_2), \\
 \alpha_5 &\equiv \frac i8(\tilde\alpha_1-\tilde\alpha_3)- \frac i8(\tilde\alpha_1-3\tilde\alpha_2).
\end{aligned}
\end{gather*}
\end{proposition}
\begin{proof}
With respect to the splitting  
\[V\otimes \Lambda^2V=V\otimes \lie{sp}(1) + V\otimes S^2E+ V\otimes \Lambda_0^2ES^2H,\]
$\alpha_1,\beta_1$ lie in the first component, $\alpha_2$ in the second, $n\alpha_3-\alpha_1$,   $n\beta_2-\beta_1$  in the third. Hence, by Lemma~\ref{lemma:formulae}, $2\tilde\alpha_1+\tilde\alpha_3-\tilde\alpha_2$ is in the image of $\partial_B$. The rest of the statement is now a straightforward computation.
\end{proof}

\section{Intrinsic torsion conditions}
\label{sec:integrability}
In this section we study the intrinsic torsion of $B$-structures, establishing formulae to compute the ``components'' of the intrinsic torsion in terms of the exterior covariant derivative of suitable invariant tensorial forms. Then we specialize to  the qc case, showing that the intrinsic torsion lies in a specific invariant subspace.

The first problem is that  $\coker\partial_B$ is not completely reducible as a $B$-module, i.e. the sequence \eqref{eqn:nonsplittingseq} does not split. To work around this problem, we shall employ a reduction to  $\Sp(n)\h^*$, which amounts to choosing an arbitrary complement of $\mathcal{D}$ at each point. Nonetheless, we are still thinking of $B$ as the structure group of qc geometry, and the main result of this section is stated in terms of $B$-structures.

Under $\Sp(n)\h^*$, we have the decomposition
\[\Lambda^2T^*\otimes T = \im(\partial_B) \oplus\Lambda^2V^*\otimes W\oplus \partial_1(W_1)\oplus\partial_2(W_2).\]
Accordingly, the torsion of a connection splits into components as
\begin{equation}
 \label{eqn:torsiondecompositionB}
\Theta = \Theta_*+\Theta^Q +\Theta_1 +\Theta_2,
\end{equation}
and the qc condition reads
\[\Theta^Q=\Theta_0.\]
Moreover, the reduction makes $\Lambda T^*$ into a bigraded vector space,
\[\Lambda T^*=\bigoplus_{p,q}\Lambda^{p,q}, \quad \Lambda^{p,q}=\Lambda^p V^*\otimes \Lambda^q W^*.\]
Given $\alpha\in \Lambda T^*\otimes S$, we shall denote by $\alpha^{p,q}$ its component in $\Lambda^{p,q}\otimes S$. This notation carries over to tensorial forms, i.e.
\[(\alpha_\theta)^{p,q} = (\alpha^{p,q})_\theta, \quad \alpha\in\Omega^{p+q}(P,S).\]
We shall also need to consider the projection
\[\pi_{-1}\colon \im\partial_B\to \Lambda^{1,1}\otimes W,\]
and set
\[\Theta_{-1}=\pi_{-1}(\Theta_*).\]

\smallskip
We can now prove a  result analogous to Proposition~\ref{prop:dsigma22}. The first step is choosing two $B$-invariant tensorial forms, namely
\begin{align*}
\eta&\in\Omega^1(P,T/V), & \eta_\theta&\equiv w^1\otimes w_1+w^2\otimes w_2+w^3\otimes w_3,\\
\gamma&\in\Omega^5(P,T/V\otimes\kappa^*), & \gamma_\theta&\equiv \omega_s\wedge w^{123}\otimes (w_s\otimes w_{123}), 
\end{align*}
 where as usual $\kappa=\Lambda^3W^*$.
A direct computation with highest weight vectors, together with Lemma~\ref{lemma:formulae}, give:
\begin{lemma}
\label{lemma:f0}
The kernel of the map
\begin{align*}
z\colon \Lambda^2T^*\otimes T&\to  \Lambda^3T^*\otimes W, \quad \alpha\to  \alpha\hook ( \omega_s\wedge w^{123})\otimes w_s
\end{align*}
contains (and equals when $n>1$)
\[\partial_B\Bigl(T\otimes (\lie{sp}(n)+\Hom(W,V))+W^*\otimes\h\Bigr)+\partial_2(W_2)+2EH+ES^3H,\]
where $2EH+ES^3H\subset \partial_B(V^*\otimes\h)+\partial_1(W_1)$ contains the highest weight vectors
\begin{gather*}
(\tilde\alpha_2+\tilde\alpha_3+8\alpha_4)+2(\tilde\alpha_1-\tilde\alpha_3)+2(\tilde\alpha_1-3\tilde\alpha_2),\\
(8i\alpha_5-\tilde\alpha_3+3\tilde\alpha_2)-(\tilde\alpha_1-\tilde\alpha_3)+(\tilde\alpha_1-3\tilde\alpha_2),\quad
(\tilde\beta_2+2i\beta_4)-\tilde\beta_2.
\end{gather*}
Regardless of $n$,  the restriction  $z|_{\partial_1(W_1)}$ is injective, and $z|_{\im\partial_B}=g\circ \pi_{-1}$, where
\[g(\alpha)=\alpha\wedge(\Theta_0\hook\sigma)+\Theta_0\wedge(\alpha\hook\sigma).\]
\end{lemma}

\begin{proposition}
\label{prop:Bintrinsic}
Le $P$ be a  $B$-structure; for every connection on $P$ and every reduction to $\Sp(n)\h^*$, 
\begin{align*}
 \Theta_2+\Theta_{-1}&=(D\eta)^{1,1},&
\Theta^Q&=(D\eta)^{2,0},\\
(\Theta_2+\Theta_{-1})\hook\sigma&=(D\sigma)^{1,3},&
\Theta^Q\hook\sigma&=(D\sigma)^{2,2},\\
\Theta_1\hook\gamma&=(D\gamma)^{3,3}-g(\Theta_{-1}), &
\Theta^Q\hook\gamma&=(D\gamma)^{4,2}.
\end{align*}
These equations determine the intrinsic torsion in the sense that in each equation (except the last one for $n=1$) the components of the intrinsic torsion appearing on the left hand side are determined by the right hand side.
\end{proposition}
\begin{proof}
By \eqref{eqn:Dalphanablaalpha} $D\eta=\Theta\hook\eta$; more precisely,
\[(D\eta)^{2,0}=\Theta^{2,0}\hook\eta, \quad (D\eta)^{1,1}=\Theta^{1,1}\hook\eta.\]
These interior products correspond respectively to the contractions 
\[\begin{split}
(\Lambda^{2,0}\otimes W)&\otimes (W^*\otimes W) \to \Lambda^{2,0}\otimes W,\\
(\Lambda^{1,1}\otimes W )&\otimes (W^*\otimes W) \to \Lambda^{1,1}\otimes W.
 \end{split}\]
Moreover, since $\eta$ is the identity in $W^*\otimes W=\Hom(W,W)$, contraction with $\eta$ gives rise to two isomorphisms. Thus, $(D\eta)^{1,1}$ determines the component $\Lambda^{1,1}\otimes W$ of the torsion, which equals $\Theta_2+\Theta_{-1}$. Similarly, the  component $\Lambda^2V^*\otimes W$ can be read off $(D\eta)^{2,0}$, and the same arguments apply to $D\sigma$.

By the same token,
\[D\gamma=\Theta\hook\gamma;\]
since $(\Theta\hook\gamma)_\theta=z(\Theta_\theta)\otimes w_{123}$, and by  Lemma~\ref{lemma:f0} the restriction of $z$ to
\[\Lambda^{2,0}\otimes W\to \Lambda^{4,2}\otimes W\]
$z$ is injective, it follows that $(D\gamma)^{4,2}=\Theta^Q\hook\gamma$ determines $\Theta^Q$ if $n>1$.

Again by Lemma~\ref{lemma:f0},
\[(D\gamma)^{3,3}=\Theta_1\hook \gamma+g(\Theta_{-1}),\]
and this equation determines $\Theta_1$ because $z$ is injective on $\partial_1(W_1)$.
\end{proof}
\begin{example}
Going back to the example $\Sp(n+1)/\Sp(n)\Sp(1)$ of Section~\ref{sec:example}, observe that  $D$ is the horizontal part of $d$; therefore,
\begin{gather*}
D\sigma=(\omega_1\wedge w^{23}+ \omega_2\wedge e^{31}+\omega_3\wedge e^{12})\otimes w_{123},\\
D\eta=\omega_s\otimes w_s-2w_s\hook w^{123}\otimes w_s,\\
D\gamma= \omega_s\wedge (\omega_1\wedge w^{23}+ \omega_2\wedge e^{31}+\omega_3\wedge e^{12})\otimes (w_s\otimes w_{123}).
\end{gather*}
Working with the reduction to $\Sp(n)\h^*$ introduced in Section~\ref{sec:example}, Proposition~\ref{prop:Bintrinsic} gives
\[\Theta_1=\Theta_2=\Theta_{-1}=0, \quad \Theta^Q=\Theta_0.\]
\end{example}

We now turn to qc geometry. Let us consider the map
\[h\colon V^*\otimes W^*\otimes W\to \Lambda^{3,0}\otimes W\]
 obtained by tensoring the identity $W\to W$ with the map
\begin{equation*}
V^*\otimes W^*\to \Lambda^{3,0}, \quad  e^i\otimes w^j\to e^i\wedge\omega_j.
\end{equation*}

\begin{lemma}
\label{lemma:twomaps} 
The kernel of 
\[f\colon \partial_1(W_1)+\partial_2(W_2)\to \Lambda^{3,0}\otimes W, \quad f(\Theta_1,\Theta_2)=z(\Theta_1)-h(\Theta_2).\]
is isomorphic to
\[\begin{cases} ES^3H+ES^5H, & n=1 \\ ES^3H, & n>1\end{cases};\]
the component $ES^3H$ contains the  highest weight vector
\[4\tilde\beta_1+\tilde\beta_2+2i\beta_3=\partial_1(-2\beta_2-3\beta_1)+2i\partial_2(\beta_3).\]
\end{lemma}
Notice that $W_1\oplus W_2$ has a unique submodule isomorphic to $ES^5H$; from now on,  $ES^5H$ will indicate this submodule unless otherwise specified. The component isomorphic to $ES^3H$ identified in the lemma will be denoted by $\widetilde{ES^3H}$, and we will denote by $\Theta_0^B$ the image of $\Theta_0$ in $\coker\partial_B$.
\begin{theorem}
\label{thm:reductiontheorem}
A $B$-structure is qc if and only if the intrinsic torsion takes values in
\[\begin{cases}\Theta_0^B + \widetilde{ES^3H}, & n>1\\
   \Theta_0^B + \widetilde{ES^3H}+ES^5H,& n=1
\end{cases}\]
\end{theorem}
\begin{proof}
Choose an arbitrary connection. The qc condition $\Theta^Q=\Theta_0$ implies
\[\gamma = \Theta_0\wedge \sigma= D\eta\wedge\sigma,\]
whence
\[D\gamma=D^2\eta\wedge\sigma+D\eta\wedge D\sigma;\]
however, if $\Omega$ is the curvature, $D^2\eta\wedge\sigma=\Omega\wedge\eta\wedge\sigma=0$, so
\[D\gamma-D\eta\wedge D\sigma=0.\]
Decomposing into components, we get
\[(D\gamma)^{3,3}-(D\eta)^{1,1}\wedge(D\sigma)^{2,2}-(D\eta)^{2,0}\wedge (D\sigma)^{1,3}=0,\]
whence, by Proposition~\ref{prop:Bintrinsic},
\[(D\gamma)^{3,3}= \Theta_2\wedge (\Theta_0\hook\sigma)+\Theta_0\wedge (\Theta_2\hook\sigma) + \Theta_{-1}\wedge(\Theta_0\hook\sigma)+\Theta_0\wedge(\Theta_{-1}\hook\sigma).\]
Up to a contraction $\Lambda^3W^*\otimes \kappa^*\cong\R$, the map 
\[\Theta_2\to \Theta_2\hook\sigma\]
corresponds to the trace $V^*\otimes W^*\otimes W\to V^*$; by construction this is zero. 
It follows that
\[(D\gamma)^{3,3}= \Theta_2\wedge (\Theta_0\hook\sigma)+g(\Theta_{-1}),\]
so by Proposition~\ref{prop:Bintrinsic} $z(\Theta_1)=h(\Theta_2)$, and the statement follows from Lemma~\ref{lemma:twomaps}.
\end{proof}
\section{A further reduction}
Theorem~\ref{thm:reductiontheorem} relies on the decomposition \eqref{eqn:torsiondecompositionB}, which depends in turn on the choice of a reduction of the structure group to $\Sp(n)\h^*$. In this section we illustrate how the choice of this reduction affects the torsion, and show that it is  canonical {\em in part}; in other words, we obtain a canonical reduction to an intermediate group $K$, \[\Sp(n)\h^*\subset K\subset B.\]
More precisely, the splitting under $\Sp(n)\h^*$
\[\Hom(W,V)=EH \times ES^3H,\]
is also a product of abelian Lie groups, so that $EH$ and $ES^3H$ appear as subgroups of $B$. Notice that this $EH$ is isomorphic to $V^*$, rather than  $V$, as an $\Sp(n)\h^*$-module; however, if we identify $V$ and $V^*$, then
\begin{equation}
 \label{eqn:explicitEH}
EH=\bigl\{w^s\otimes v\hook\omega_s\mid v\in V\bigr\}.
\end{equation}
We set
\[K=\Sp(n)\h^*\ltimes EH.\]
We will see that an arbitrary $B$-structure has a canonical $K$-reduction induced by the choice of a complement of $\widetilde{ES^3H}$ in $W_1+W_2$, but in the qc case, thanks to Theorem~\ref{thm:reductiontheorem}, the reduction is independent of the choice of complement of $\widetilde{ES^3H}$.

The key fact is that the action of the subgroup $\Hom(W,V)\subset B$ on $\Theta_0$  is ``linearized'' when taking the quotient by $\im\partial_B$, as shown in the following lemma.
\begin{lemma}
\label{lemma:linearized}
The Lie group $\Hom(W,V)$ acts on $\Theta_0^B$ with stabilizer equal to $EH$; the orbit is $\Theta^B_0+\widetilde{ES^3H}$.
 \end{lemma}
\begin{proof}
The action of the Lie group $\Hom(W,V)$ induces an infinitesimal action of its Lie algebra, which coincides with $\Hom(W,V)$ itself. Denoting the former action by juxtaposition and the latter by $\cdot$, we see that
\[w^i\otimes e_j\cdot \alpha = -w^i\wedge e_j\hook\alpha, \quad \alpha\in\Lambda^2V^*.\]
Therefore
\[g\Theta_0=\Theta_0 + g\cdot\Theta_0 \mod \Lambda^{0,2}\otimes T + \Lambda^{1,1}\otimes V\subset \im\partial_B,\]
i.e. $g\Theta_0^B=\Theta_0^B + g\cdot\Theta_0^B$.
The Lie algebra action gives an $\Sp(n)\Sp(1)$-invariant map
\[\Hom(W,V)\otimes \Lambda^2T^*\otimes T\to \Lambda^2T^*\otimes T;\]
since $\Theta_0$ is invariant, the action on $\Theta_0$ gives an invariant map
\[\Hom(W,V)\to \Lambda^2T^*\otimes T.\]
Under this map, the highest weight vector 
\[(w^2+iw^3)\otimes v_1h_2\in ES^3H\subset \Hom(W,V)\]
has image 
\[\tilde\beta_1+\frac i2\beta_3 -\frac i2\beta_4\equiv \tilde\beta_1+\frac i2\beta_3 +\frac 14\tilde\beta_2,\]
which by definition lies in $\widetilde{ES^3H}$.

On the other hand the highest weight vector 
\[(w^2+iw^3)\otimes v_1h_1 + i w^1\otimes v_1h_2\in EH\subset \Hom(W,V)\]
has image  
\begin{multline*}
(\omega_2+i\omega_3)\otimes v_1h_1+i\omega_1\otimes v_1h_2 -v_1h_2\wedge w^s\otimes w_s\\
+iv_1h_1 \wedge (w^1\wedge(w_2+iw_3))+iv_1h_2\wedge(w_2\wedge w_3)
=\frac12\tilde\alpha_1-\alpha_4+i\alpha_5,
\end{multline*}
which by Proposition~\ref{prop:basis} lies in $\im\partial_B$.
\end{proof}

\begin{theorem}
\label{thm:canonicalKreduction}
For any fixed $\Sp(n)\h^*$-invariant complement $\widetilde{ES^3H}^\perp$ of $\widetilde{ES^3H}$ in $W_1+W_2$, every $B$-structure has a unique $K$-reduction such that the restriction of the $B$-intrinsic torsion takes values in $\Theta_0^B+\widetilde{ES^3H}^\perp$.
\end{theorem}
\begin{proof}
Let $\Theta^B\colon P\to\coker\partial_B$ be the intrinsic torsion. By Lemma~\ref{lemma:linearized}, $\Theta_0^B$ has stabilizer $K$; consider the $K$-equivariant map 
\[f\colon P\to\coker\partial_B, \quad f(u)=\Theta^B(u)-\Theta_0^B.\]
Since $\Theta^Q(u)=r(\Theta^B(u))=\Theta_0^Q$,  the map $f$ takes values in $\ker r$, which by Proposition~\ref{prop:kernelr} equals $W_1+W_2$.

Now set 
\[\tilde P=\{u\in P\mid f(u)\in \widetilde{ES^3H}^\perp\}.\]
By construction, $\tilde P$ is closed under the action of $K$. Conversely, given any $u\in P$, $g\in\Hom(W,V)$,
\[f(ug)=\Theta^B(ug)-\Theta_0^B = g^{-1}(\Theta^B(u)-g\Theta_0^B)\]
lies in  $\widetilde{ES^3H}^\perp$ if and only if so does $\Theta^B(u)-g\Theta_0^B$. By  Lemma~\ref{lemma:linearized}, this condition is satisfied for exactly one $g\in ES^3H$. Thus, $\tilde P$ is a $K$-structure.
\end{proof}

Combining this result with Theorem~\ref{thm:reductiontheorem}, and specializing to the qc case, we find
\begin{corollary}
\label{cor:Kreduction}
Every qc $B$-structure has a unique $K$-reduction $P$ such that 
\[\begin{cases}\Theta^B(u)=\Theta_0^B, & n>1 \\ \Theta^B(u)\in \Theta_0^B+ES^5H,& n=1\end{cases}\]
for all $u$ in $P$. 
\end{corollary}
We shall refer to such a $K$-structure as a \dfn{qc $K$-structure}. We shall say a qc $K$-structure $P$ is \emph{integrable} if $\Theta^B=\Theta^B_0$ identically on $P$. This condition is automatic when $n>1$. We shall see in Corollary~\ref{cor:integrableisintegrable} that this definition agrees with that of  \cite{Duchemin}.

\begin{remark}
An alternative approach more akin to Biquard's would be to fix a compatible metric on $\mathcal{D}$; in other words, by choosing an arbitrary reduction from $B$ to  $\tilde B=\Sp(n)\Sp(1)\ltimes \Hom(W,V)$. Proceeding as in Lemma~\ref{lemma:linearized}, we would see that $\im\partial_{\tilde B}$ does not contain
$\frac12\tilde\alpha_1-\alpha_4+i\alpha_5$, and so the construction of Corollary~\ref{cor:Kreduction} would give a ``canonical'' reduction to $\Sp(n)\Sp(1)$, depending only on the choice of the metric.
However, as we are mainly interested in the intrinsic geometry of qc structures, we will refrain from fixing a metric until Section~\ref{sec:posit}.
\end{remark}
\begin{remark}
Since the projection of $\widetilde{ES^3H}$ onto $W_2$ is injective, the $K$-reduction is characterized by the condition $\Theta_2=0$, or $\Theta_2\in ES^5H$ when $n=1$. In practice, if one finds a connection which satisfies this condition with respect to some frame, then the frame is compatible with the $K$-reduction of Corollary~\ref{cor:Kreduction}.
\end{remark}

\section{$K$-intrinsic torsion}
\label{sec:kintrinsictorsion}
Corollary~\ref{cor:Kreduction} motivates us to study the intrinsic torsion of qc $K$-structure. In fact, we have defined integrable qc $K$-structures by the condition that the restriction of the $B$-intrinsic torsion be constant, equal to $\Theta_0^B$. In this section we show that the $K$-intrinsic torsion of an integrable qc $K$-structure is also constant. 

It follows that a qc $K$-structure admits a family of connections whose torsion equals $\Theta_0$; this family is parametrized by sections of a bundle with fibre $\ker\partial_K$, which is identified by the following:
\begin{lemma}
\label{lemma:partialK}The kernel of $\partial_K$ is the $S^2H$ containing
\begin{multline}
\label{eqn:kerpartialK}
 v_jh_2\otimes ((w^2+iw^3)\otimes v_{n+j}h_1 + iw^1\otimes v_{n+j}h_2))
- v_{n+j}h_2\otimes ((w^2+iw^3)\otimes v_{j}h_1 + iw^1\otimes v_{j}h_2))\\
-(w^2+iw^3)\otimes(\id_V+2\id_W) + iw^1\otimes (\omega_2+i\omega_3)-i(w^2+iw^3)\otimes \omega_1 ,
\end{multline}
and the inclusion $V^*\otimes \frac{\lie{b}}{\lie{k}}\to T^*\otimes\frac{\lie{b}}{\lie{k}}$ induces an exact sequence of $K$-modules
\begin{equation}
 \label{eqn:cokerKB}
0\to S^2E S^2H\to V^*\otimes \frac{\lie{b}}{\lie{k}}\to \coker\partial_K\to\coker\partial_B\to 0
\end{equation}
\end{lemma}
\begin{proof}
If we set
\[T_1=W^*\otimes\Hom(W,V), \quad T_2=W^*\otimes \Hom(\lie{sp}(n)+\h)+V^*\otimes \Hom(W,V),\]
Corollary~\ref{cor:kerpartialB} implies $\ker\partial_B$ decomposes as the direct sum of $S^2W^*\otimes V\subset T_1$ and $S^2H(S^2E+\R)\subset T_2$.
Since $\partial(T_1)$ and $\partial(T_2)$ intersect trivially, $\ker\partial_K$ is also a direct sum, i.e.  
\[\ker\partial_K = \bigl(S^2W^*\otimes V\cap (T^*\otimes\lie{k})\bigr) + \bigl(S^2H(S^2E+\R)\cap  (T^*\otimes\lie{k})\bigr).\]
The first component is contained in  $W^*\otimes EH=E(H+S^3H)$, which contains 
\[(w_2+iw_3)\otimes (-iw_1\otimes v_1h_1-(w_2-iw_3)\otimes v_1h_2) +i w_1\otimes ((w^2+iw^3)\otimes v_1h_1 + i w^1\otimes v_1h_2),\]
\[(w_2+iw_3)\otimes(w_2+iw_3)\otimes v_1h_1 + i (w_2+iw_3)\otimes w^1\otimes v_1h_2.\]
Neither of these vectors is in the kernel of $\partial_K$, so $\ker\partial_K$ intersects $T_1$ trivially.

Now the components of type $S^2HS^2E$ inside $T_2\cap  (T^*\otimes\lie{k}) $ are identified by
\[(w^2+iw^3)\otimes (v_1h_2\wedge v_1h_1), \quad v_1h_2\otimes ((w^2+iw^3)\otimes v_1h_1 + i w^1\otimes v_1h_2),\]
so $\ker\partial_K$ contains no $S^2HS^2E$.
The $S^2H$ components inside $T_2\cap  (T^*\otimes\lie{k})$ contain 
\begin{gather*}
(w^2+iw^3)\otimes(\id_V+2\id_W), \quad w_1\otimes (\omega_2+i\omega_3)-(w^2+iw^3)\otimes \omega_1,\\
 v_jh_2\otimes ((w^2+iw^3)\otimes v_{n+j}h_1 + iw^1\otimes v_{n+j}h_2))
- v_{n+j}h_2\otimes ((w^2+iw^3)\otimes v_{j}h_1 + iw^1\otimes v_{j}h_2)). 
\end{gather*}
It is straightforward to verify that there is only one linear combination that goes to zero, up to multiple, namely \eqref{eqn:kerpartialK}.

The commutative diagram with exact rows 
\[\xymatrix{
0\ar[r]&T^*\otimes\lie{k}\ar[r]\ar[d]^{\partial_{K}} & T^*\otimes\lie{b}\ar[d]^{\partial_{B}}\ar[r]& T^*\otimes\frac{\lie{b}}{\lie{k}}\ar[r]\ar[d]&0\\
0\ar[r]& \Lambda^2T^*\otimes T\ar[r]&\Lambda^2T^*\otimes T\ar[r]& 0
 }\]
determines an exact sequence
\[0\to \ker\partial_K\to \ker\partial_B\to  T^*\otimes\frac{\lie{b}}{\lie{k}}\xrightarrow{f} \coker\partial_K\to\coker\partial_B\to 0.\]
By the exactness of the top row in the diagram, $S^2W^*\otimes V$ maps injectively into  $T^*\otimes\frac{\lie{b}}{\lie{k}}$; the image equals $W^*\otimes\frac{\lie{b}}{\lie{k}}$ by a dimension count. This gives exactness of \eqref{eqn:cokerKB} as a sequence of vector spaces. Moreover, we have a diagram
\[\xymatrix{V^*\otimes\frac{\lie{b}}{\lie{k}}\ar[r]^\iota\ar[dr]^{\id} & T^*\otimes\frac{\lie{b}}{\lie{k}}\ar[d]\ar[r]^f & \coker\partial_K \\ & V^*\otimes\frac{\lie{b}}{\lie{k}}\ar[ur]^{\tilde f}}\]
where $\tilde f$ is $K$-equivariant. Therefore, $f\circ\iota$ is also equivariant.

Since  $EH$ acts trivially on $V^*\otimes\frac{\lie{k}}{\lie{b}}$, all maps in \eqref{eqn:cokerKB} are equivariant.
\end{proof}

Working with   $\Sp(n)\Sp(1)$-modules, we can define a complement of $S^2ES^2H$ in $V^*\otimes ES^3H\subset V^*\otimes\lie{b}$, namely 
\[W_3=\begin{cases}S^4H(S^2E+\Lambda^2_0E+\R)+S^2H\Lambda^2_0E+S^2H, & n>1 \\  S^4H(S^2E+\R)+S^2H, & n=1\end{cases},\]
and by the above lemma we have a decomposition into $\Sp(n)\Sp(1)$-modules
\[\Lambda^2T^*\otimes T = \im(\partial_K) \oplus\Lambda^2V^*\otimes W\oplus \partial_1(W_1)\oplus\partial_2(W_2)\oplus \partial(W_3).\]
Accordingly, the torsion of a connection splits into components as
\[\Theta=\Theta_*+\Theta^Q+\Theta_1+\Theta_2+\Theta_3.\]
This decomposition is not a decomposition of $K$-modules, i.e. it depends on the pointwise choice of a complement of $\mathcal{D}$. But restricting to the qc case, we find the following:
\begin{lemma}
\label{lemma:Kintrinsictorsion}
The intrinsic torsion of a qc $K$-structure takes values in
\begin{equation}
\label{eqn:theta0k}
\begin{cases}
\Span{\Theta_0^K}\oplus \partial(W_3), & n>1 \\ 
 \Span{\Theta_0^K}\oplus \partial(W_3)\oplus ES^5H, & n=1 \\ 
\end{cases}
\end{equation}
which is a direct sum of $K$-submodules of $\coker\partial_K$.
\end{lemma}
\begin{proof}
Having defined qc $K$-structures by the condition of Corollary~\ref{cor:Kreduction}, it suffices to show that \eqref{eqn:theta0k} is a $K$-module. In \eqref{eqn:cokerKB}, $EH$ acts trivially on $V^*\otimes\frac{\lie{k}}{\lie{b}}$, so $\partial(W_3)$ is a $K$-submodule. The same applies to $ES^5H\subset \Lambda^{1,1}\otimes W$. It remains to be seen how $EH$ acts on $\Theta_0^K$.
The group action of $\Hom(W,V)$ on $\Lambda V^*$ is given on simple elements by
\[(w^i\otimes e_j)\alpha = \alpha-w^i\wedge e_j\hook\alpha.\]
Recalling \eqref{eqn:explicitEH},
\begin{multline*}
(w^s\otimes v\hook\omega_s) \omega_r\otimes w_r
=\Theta_0 + ( w^s\otimes v\hook\omega_s) \cdot\Theta_0
-\norm{v}^2\eta\hook\sigma\\
-w^s\wedge ((v\hook\omega_s)\hook\omega_r)\otimes (v\hook\omega_r)
- \norm{v}^2w_r\hook w^{123}\otimes (v\hook\omega_r).
\end{multline*}
However $ ( w^s\otimes v\hook\omega_s) \cdot\Theta_0$ was seen in Lemma~\ref{lemma:linearized} to be in some $EH\subset\im\partial_B$, but it is really in $\im\partial_K$, since $V^*\otimes ES^3H$ contains no $EH$.
Similarly, $\eta\hook\sigma$ lies in a trivial submodule of $\Lambda^{0,2}\otimes W$, which is contained in $\im\partial_K$, and 
\[w_r\hook w^{123}\otimes (v\hook\omega_r)\in \Lambda^{0,2}\otimes V\subset\im\partial_K.\]
Thus,
\[
( w^s\otimes v\hook\omega_s) \omega_r\otimes w_r-\Theta_0=-w^s\wedge ((v\hook\omega_s)\hook\omega_r)\otimes (v\hook\omega_r) \mod \im\partial_K,\]
which is in the image of $S^2(EH)$ under the equivariant map $V\otimes V\to \Lambda^{1,1}\otimes V$
\[
u\otimes v\to -\bigl(w^1\wedge ((u\hook\omega_1)\hook\omega_r)+w^2\wedge ((u\hook\omega_2)\hook\omega_r)+w^3\wedge ((u\hook\omega_3)\hook\omega_r)	\bigr)\otimes (v\hook\omega_r).
 \]
The target space $\Lambda^{1,1}\otimes V$ is contained in $\partial(W_3)+\im\partial_K$; however, $\partial(W_3)$ has no irreducible component in common with  $S^2(EH)=S^2ES^2H+\Lambda^2_0E+\R$. We conclude that the action of $EH$ on $\Theta_0^K$ is trivial.
\end{proof}

A  computation with highest weight vectors yields:
\begin{lemma}
\label{lemma:q}
Consider the map
\[q\colon \Lambda^{2,0}\otimes \h \to \Lambda^{2,1}\otimes W, \quad q(\alpha\otimes p) =\alpha \wedge w^s\otimes \widetilde{\ad}(p)w_s;\]
then
\[\left\{v\in W_3\mid \partial(v)\hook\Theta_0\in\im q\right\}=0.\]
\end{lemma}
We can now prove the main result of this section.
\begin{theorem}
\label{thm:Kintrinsictorsion}
The intrinsic torsion of an integrable qc $K$-structure is constant, i.e. 
\[\Theta^K\equiv\Theta^K_0.\]
\end{theorem}
\begin{proof}
Choose an arbitrary reduction to $\Sp(n)\Sp(1)$ and a connection such that the torsion takes values in $\partial(W_3)+\Theta_0$. In particular, $\Theta_{-1}$ and $\Theta_2$ vanish, as does the component $\Theta^{0,2}$. Then the Bianchi identity and Proposition~\ref{prop:Bintrinsic} give
\[\Omega\wedge\eta=D^2\eta=D\Theta_0=\Theta\hook\Theta_0.\]
In particular,
\[\Omega^{2,0}\wedge\eta=\Theta^{1,1}\hook\Theta_0.\]
By Lemma~\ref{lemma:q}, this implies that $\Theta^{1,1}$ is zero.
\end{proof}

Since the intrinsic torsion is trivial, we are motivated to consider specific connections. As  $\Theta_0$ is not invariant under $K$, we cannot conclude that there exists a connection with torsion equal to $\Theta_0$: we must consider the stabilizer of $\Theta_0$ in $\Lambda^2T^*\otimes T$ and take a corresponding reduction.

Given a $K$-structure $P$ and a reduction $\tilde P$ to $\Sp(n)\h^*$, we define a \dfn{$K$-connection} on $\tilde P$ as a one-form $\omega\in\Omega^1(\tilde P,\lie{k})$ which is the restriction of a (uniquely determined) connection form on $P$.

\begin{corollary}
\label{cor:canonicalconnection}
Any reduction to $\Sp(n)\h^*$ of an integrable qc $K$-structure  admits a $K$-connection with torsion equal to $\Theta_0$.
\end{corollary}
\begin{proof}
Fix an arbitrary connection $\omega$, and let $\Theta$ be its torsion. By Theorem~\ref{thm:Kintrinsictorsion}, $\Theta_\theta-\Theta_0$  defines an $\Sp(n)\h^*$-equivariant
map taking values  in $\im \partial_K$. Since $\partial_K\colon T^*\otimes\lie{k}\to\im\partial_K$ admits an $\Sp(n)\h^*$-invariant right inverse, we obtain  a tensorial $\lie{k}$-valued $1$-form  $A$ such that the connection $\omega+A$ has torsion  $\Theta_0$.
 \end{proof}
This result only applies to integrable qc structures. For this reason, in the rest of the paper we shall only consider integrable qc structures.

\begin{remark}
Given a reduction  to $\Sp(n)\h^*$, all other reductions are parametrized by forms of type $(1,0)$, since $K/\Sp(n)\h^*\cong V^*$ as an $\Sp(n)\h^*$-module. Such a reduction is determined canonically once one fixes a compatible metric on $\mathcal{D}$ (see \cite{Biquard}, or Proposition~\ref{prop:metricdeterminescomplement}); replacing the metric $g$ with a different metric $fg$ in the conformal class affects the reduction  to $\Sp(n)\h^*$ via  $df^{1,0}$.
\end{remark}

\section{Connections and curvature}
\label{sec:curvature}
We saw in Corollary~\ref{cor:canonicalconnection} that given an integrable qc $K$-structure,  any reduction to $\Sp(n)\h^*$ has a $K$-connection with torsion equal to $\Theta_0$, though this is not unique due to the fact that $\partial_K$ has a kernel. In this section we find a natural condition that can be imposed on the curvature of such a connection that makes it unique, given the choice of the reduction. As a byproduct, we
compute the $K$-module in which the curvature lies. Notice that this is not a metric connection, unlike the Biquard connection.

The curvature of a connection with torsion $\Theta_0$ takes values in
$\Lambda^2T^*\otimes \lie{k}$;
this is a space of dimension  $16 n^4+O(n^3)$
which, as an $\Sp(n)\Sp(1)$-module, decomposes as 
\[\begin{split}
S^4E
+ V_{31}(\R+S^2H)
+V_{22}
+V_{211}S^2H
+ S^3E(2H+S^3H) 
+V_{21}(3H+2S^3H)\\ 
+\Lambda^3_0E(H+S^3H)
+ S^2 E(3\R+6S^2H+S^4H)
+\Lambda^2_0E(3\R+5S^2H+2S^4H)\\
+E(8H+7S^3H+S^5H)
+ 4\R+6S^2H +3S^4H
\end{split}\]
However,  the Bianchi identity implies that the curvature lies in a module of dimension $\frac23n^4+O(n^3)$, identified by the following:
\begin{lemma}
\label{lemma:Kcurvature}
Let \[\delta\colon \Lambda^2T^*\otimes \lie{k}\to\Lambda^3\otimes T\]
be the restriction of the skew-symmetrization map $\Lambda^2T^*\otimes \gl(T)\to\Lambda^3\otimes T$, and set
\begin{equation*}
s\colon T^*\otimes (\Lambda^{2}T^*\otimes T)\to \Lambda^3T^*\otimes T, \quad x\otimes y\otimes z \to x\wedge y\otimes z.
\end{equation*}
If we denote by $\widetilde{EH}$ the submodule  of $\Lambda^{2}T^*\otimes T$ containing
\begin{equation}
\label{eqn:generatesEH}
 \frac12\tilde\alpha_1-\alpha_4+i\alpha_5\in\Lambda^{2,0}\otimes V+\Lambda^{1,1}\otimes W,
\end{equation}
the preimage of  $s(T^*\otimes \widetilde{EH})$ under $\delta$ is the direct sum of
\[\begin{alignedat}{2}
R_1&\cong S^4E+ (S^2E+\Lambda^2_0E+\R)(S^2H+\R)&\subset \Lambda^{2,0}\otimes(\lie{sp}(n)+\h)\\
R_2&\cong S^3EH+2ES^3H+2EH &\subset \Lambda^{1,1}\otimes(\lie{sp}(n)+\h)+\Lambda^{2,0}\otimes EH\\
R_3&\cong S^2ES^2H +S^4H+S^2H+\R&\subset \Lambda^{0,2}\otimes(\lie{sp}(n)+\h)+\Lambda^{1,1}\otimes EH\\
R_4&\cong ES^3H&\subset \Lambda^{0,2}\otimes EH.
\end{alignedat}\]
Moreover, denoting by $\tr\colon\Lambda^{2}T^*\otimes \lie{k}\to\Lambda^2T^*$ the trace,
\[\begin{alignedat}{1}
\tilde R_1=R_1\cap \ker \tr&\cong S^4E+ S^2ES^2H+\Lambda^2_0E+\R,\\
\tilde R_2=R_2\cap \ker \tr&\cong S^3EH+ES^3H+EH,\\
\tilde R_3=R_3\cap \ker \tr&\cong  S^2ES^2H +S^4H+\R.\\
\end{alignedat}\]
\end{lemma}
\begin{proof}
To begin with, we determine the  kernel of $\delta$, which we decompose as the direct sum of the kernels of the following maps, obtained by restriction:
\[\begin{split}
 \delta_1\colon\Lambda^{2,0}\otimes(\lie{sp}(n)+\h)&\to \Lambda^{3,0}\otimes V+\Lambda^{2,1}\otimes W\\
\delta_2\colon\Lambda^{1,1}\otimes(\lie{sp}(n)+\h)+\Lambda^{2,0}\otimes EH&\to \Lambda^{2,1}\otimes V+\Lambda^{1,2}\otimes W\\
\delta_3\colon\Lambda^{0,2}\otimes(\lie{sp}(n)+\h)+\Lambda^{1,1}\otimes EH&\to \Lambda^{1,2}\otimes V+\Lambda^{0,3}\otimes W\\
\delta_4\colon\Lambda^{0,2}\otimes EH&\to \Lambda^{0,3}\otimes V
   \end{split}\]
The kernel of $\delta_1$ is contained in $\Lambda^{2,0}\otimes(\lie{sp}(n))$; by construction, it coincides with the space of curvature tensors of metrics with holonomy $\Sp(n)$, i.e. the kernel of
\[S^2(\lie{sp}(n))\to \Lambda^4T^*, \quad a\odot b\to a\wedge b,\]
which is known to equal $S^4E$ (see \cite{RozanskyWitten}).

The restriction of $\delta_3$ to $\Lambda^{1,1}\otimes EH+\Lambda^{0,2}\otimes\R$ is an isomorphism: in fact, $\delta_3(\Lambda^{11}\otimes EH)$ coincides with $\Lambda^{1,0}\otimes \partial(W^*\otimes EH)$, and we know from Lemma~\ref{lemma:partialK} that $\partial_K$ is injective. Thus, $\ker \delta_3\cong S^2ES^2H +\R+S^2H+S^4H$.

The kernel of $\delta_4$ is $ES^3H$, because it is clearly surjective.

Considering the composition of $\delta_2$ with the projection on $\Lambda^{1,2}\otimes W$, we see that 
$\delta_2$ has the same kernel as its restriction to 
\[\Lambda^{1,1}\otimes\lie{sp}(n)+ \Lambda^{1,0}\wedge S^2H + \Lambda^{2,0}\otimes EH,\]
where $S^2H$ contains 
\[w^1\otimes (\omega_2+i\omega_3)-(w^2+iw^3)\otimes \omega_1  + i(w^2+iw^3)\otimes(\id_V+2\id_W).\]

Moreover the restrictions of $\delta_2$ to $\Lambda^{1,1}\otimes\lie{sp}(n)+ \Lambda^{1,0}\wedge S^2H$ and $ \Lambda^{2,0}\otimes EH$ are injective. Thus, we must investigate the common components in 
\[\begin{cases}S^3E(S^3H+H)+V_{21}(S^3H+H)+2E(S^3H+H), &n>1 \\  S^3E(S^3H+H)+2E(S^3H+H), &n=1\end{cases}\]
and
 \[\begin{cases}S^3EH+V_{21}(2H+S^3H)+\Lambda^3_0E(H+S^3H)+E(3H+2S^3H), &n>1 \\S^3EH+E(2H+S^3H),&n=1\end{cases} \]
Decomposing the target space $\Lambda^{2,1}\otimes V$ as 
\[\begin{cases} \begin{aligned} S^3E(S^3H&+H)+  V_{21}(3H+3S^3H+S^5H)\\ &+\Lambda^3_0E(2H+2S^3H+S^5H)+E(5S^3H+5H+2S^5H )\end{aligned}& n>1\\ S^3E(S^3H+H)+E(3S^3H+3H+S^5H )& n=1 \end{cases}\] 
we see that the kernel contains exactly one copy of $S^3EH$.

Considering the components isomorphic to $ES^3H$ in $\Lambda^{1,1}\otimes\lie{sp}(n)$ and $\Lambda^{1,0}\wedge S^2H$, one verifies that only the second among
$(w^2+iw^3)\wedge\alpha_2$ and
\begin{equation}
\label{eqn:ES3HinR2}
v_1h_2\wedge w^1\otimes (\omega_2+i\omega_3)-v_1h_2\wedge(w^2+iw^3)\otimes \omega_1  + iv_1h_2\wedge (w^2+iw^3)\otimes(\id_V+2\id_W)
 \end{equation}
has image in $\delta_2(\Lambda^{2,0}\otimes EH)$. Thus, the kernel contains exactly one copy of $ES^3H$. Similarly, only one of the two components isomorphic to $EH$  in $\Lambda^{1,1}\otimes\lie{sp}(n)+\Lambda^{1,0}\wedge S^2H$ has image contained in $\delta_2(\Lambda^{2,0}\otimes EH)$; this component is contained in $\Lambda^{1,0}\wedge S^2H$ and is identified by the
 highest weight vector
\begin{equation}
\label{eqn:EHinR2}
\begin{alignedat}{2}
v_1h_1\wedge \bigl(w^1\otimes (&\omega_2+i\omega_3)+ (w^2+iw^3)\otimes (-\omega_1  + i(\id_V+2\id_W))\bigr)\\
+v_1h_2\wedge \bigl(&\frac12i (w_2+iw_3)\otimes (\omega_2-i\omega_3)-  w^1\otimes(\id_V+2\id_W)\\
-&\frac12i (w_2-iw_3)\otimes (\omega_2+i\omega_3)\bigr). 
\end{alignedat}
\end{equation}
A long yet straightforward computation shows that for the remaining modules, $\delta_2$ has the greatest rank that Schur's lemma allows, so 
\[\ker \delta_2=S^3EH+ES^3H+ EH.\]

It follows from Lemma~\ref{lemma:formulae} that $\widetilde {EH}$ is contained in $\im\partial_K$; this implies that the image of $s$ is contained in the image of $\delta$. Moreover, $s$ is injective, and the image of  $W^*\otimes \widetilde{EH}\cong EH+ES^3H$ is contained in $\Lambda^{2,1}\otimes V+\Lambda^{1,2}\otimes W$, giving 
\[R_2=\ker\delta_2 + EH+ES^3H.\]
Similarly, we see that
\[R_1=\ker\delta_1 + S^2E(S^2H+\R)+\Lambda^2_0E(S^2H+\R)+(S^2H+\R).\]
Finally, consider the commutative diagram
\[\xymatrix{ R_1\ar[rr]^{\tr}\ar[dr]^{\delta_1} && \Lambda^{2,0}\\ V\otimes\widetilde{EH}\ar[r]^s & \Lambda^{3,0}\otimes V+\Lambda^{2,1}\otimes W\ar[ur]^f &}\]
where $f$ is induced by the contraction $W\otimes W\to\R$. Since $f\circ s$ is surjective and $\im\delta_1$ contains $\im s$, it follows that $\tr=f\circ\delta_1$ is also surjective, proving that $R_1$ is isomorphic to $\tilde R_1+\Lambda^{2,0}$.

Similarly, the trace map
\[\tr\colon \Lambda^{1,1}\otimes(\lie{sp}(n)+\h)\to \Lambda^{1,1}\cong EH+ES^3H,\]
does not kill either  \eqref{eqn:ES3HinR2} or \eqref{eqn:EHinR2}, so the restriction to $R_2$ is surjective.

In the same way, the trace maps $R_3$ to $\Lambda^{0,2}\cong S^2H$, so $\tilde R_3$ differs from $R_3$ at most by an $S^2H$. A non-zero element
\[x=x_1+x_2, \quad  x_1\in \Lambda^{0,2}\otimes\lie{sp}(1), \quad x_2\in \Lambda^{1,1}\otimes EH,\]
can only be in the kernel of $\delta$ if $x_1$ is not zero; however, if $x_1$ is in the submodule isomorphic to $S^2H$, then
$\delta(x_1)^{0,3}$ is not zero. It follows that there is no $S^2H$ in $\tilde R_3$.
 \end{proof}

\begin{proposition}
\label{prop:Bianchi}
On an integrable qc manifold, the curvature of a connection with torsion $\Theta_0$ takes values in $R_1+R_2+R_3+R_4$.
\end{proposition}
\begin{proof}
Writing $\Theta_0$ for $\langle \Theta_0, \frac12\theta\wedge\theta\rangle$, the Bianchi identity gives
\[D\Theta_0 = \Omega\wedge\theta.\]
By \eqref{eqn:Dalphanablaalpha},
\[D\Theta_0 = \langle \nabla\Theta_0,\frac16\theta\wedge\theta\wedge\theta\rangle + \Theta_0\hook\Theta_0.\]
The calculations of Lemma~\ref{lemma:linearized} show that the infinitesimal action of $\lie{k}$ on $\Theta_0$ gives an $EH$ containing 
\eqref{eqn:generatesEH}. 
 Thus $D\Theta_0$ lies in the image of $T^*\otimes  \widetilde{EH}$ under the map $s$ of Lemma~\ref{lemma:Kcurvature}.

It now follows from the Bianchi identity  that $\Omega_\theta$ is in the preimage of $s(T^*\otimes \widetilde{EH})$ under $\delta$;  Lemma~\ref{lemma:Kcurvature} concludes the proof.
\end{proof}

We can now ask  whether among the connections with  torsion $\Theta_0$ there is one with ``minimal'' curvature. Since the connection is well defined up to a section of a bundle with fibre $S^2H$, minimality should be taken to mean that the component in one of the two  $S^2H$ appearing in $R_1+R_2+R_3+R_4$ is zero. The key observation is that taking the interior product with $\Theta_0$ of a highest weight vector in $\ker\partial_K$ gives
\[-(\omega^2+i\omega^3)\otimes(\id_V+2\id_W) + i\omega^1\otimes (\omega_2+i\omega_3)-i(\omega^2+i\omega^3)\otimes \omega_1,\]
which is an element of $S^2H\subset R_1$. This is therefore the natural candidate as a component of curvature to kill.

In fact, it turns out that this component behaves like torsion, i.e. it depends on the choice of connection via a pointwise isomorphism $\ker\partial_K\cong S^2H$, although of course the full curvature tensor depends on the connection in a more complicated way. This enables us to prove:
\begin{theorem}
\label{thm:hegelianconnection}
Any $\Sp(n)\h^*$-reduction of an integrable qc $K$-structure has a unique $K$-connection such that
\begin{enumerate}
\item the  torsion is $\Theta_0$;
\item the curvature has no component in $S^2H\subset R_1$.
\end{enumerate}
Condition (2) can be replaced with
\begin{enumerate}
 \item[2'.] the curvature $\Omega$ satisfies $(\tr \Omega^{2,0})_\theta\in S^2E+\Lambda^2_0S^2E$.
\end{enumerate}
\end{theorem}
\begin{proof}
Let $\omega$ be any connection with torsion $\Theta_0$. By Proposition~\ref{prop:Bianchi}, the curvature lies in $R_1+R_2+R_3+R_4$. By Lemma~\ref{lemma:Kcurvature}, 
the component  of $R_1$ isomorphic to $S^2H$ is not contained in $\tilde R_1$, so the conditions (2) and (2') are equivalent.

The generic connection with torsion $\Theta_0$ has the form $\omega_A=\omega+A$, where $A_\theta$ takes values in $\ker\partial_K$.
The curvature of $\omega_A$ is 
\[\Omega_A=\Omega + DA+ \frac12[A,A],\]
where $\Omega$ denotes the curvature of $\omega$ and $D$ is the exterior covariant derivative with respect to $\omega$.

By Lemma~\ref{lemma:partialK},
 $A^{1,0}$ takes values in the abelian subalgebra $EH$, so
\[(\Omega_A)^{2,0}=\Omega^{2,0} + (DA)^{2,0},\]
and, by \eqref{eqn:Dalphanablaalpha}, \[DA=\frac12\langle \theta\wedge \theta, \nabla A\rangle +\Theta_0\hook A.\] 
However,  the infinitesimal action of $\lie{k}$ on $T^*\otimes\lie{k}$ takes $\ker\partial_K=S^2H$ into \[\ker\partial_K+ W^*\otimes \lie{k};\] 
it follows that $\langle \theta\wedge \theta, \nabla A\rangle^{2,0}$ is trace-free, and
\[\tr(\Omega_A)^{2,0}-\tr\Omega^{2,0}= \tr(\Theta_0\hook A).\]
By Proposition~\ref{prop:Bianchi}, the right-hand side lies in $S^2H\subset\Lambda^{2,0}$. Since the map
\[\ker\partial_K\ni t\to \tr(\Theta_0\hook t) \in S^2H\subset\Lambda^{2,0}\]
is easily seen to be an isomorphism, there is a unique $\omega_A$ such as in the statement.
\end{proof}
We shall refer to the connection of Theorem~\ref{thm:hegelianconnection} as the \emph{qc connection}.

\begin{example}
Let us consider the sphere, with the connection $\omega$ considered in Section \ref{sec:example}. Then the connection
\[\omega_{qc}=\omega+A, \quad A= -e^a \otimes w^s\otimes (e_a\hook \omega_s)\]
has torsion $\Theta_0$ and curvature
\[\Omega_{qc}=d\omega+\frac12[\omega,\omega] + DA+\frac12[A,A].\]
For $X,Y$ in $\lie{m}$, we find
\[\Omega_{qc}(X,Y)=\Omega(X,Y) + dA(X,Y)+\frac12[A,A](X,Y);\]
since $EH$ is abelian $[A,A]=0$, and 
we compute
\begin{multline*}
 \Omega_{qc}=-\sum_{a<b} e^{ab}\otimes e^a\wedge e^b - \sum_{a<b,s} e^{ab}\otimes e_a\hook\omega_s \wedge e_b\hook\omega_s -(\omega_s-2w_s\hook w^{123})\otimes \omega_s\\
-\sum_{a,s,r} e_a\hook\omega_r\wedge w^r \otimes w^s\otimes (e_a\hook \omega_s).
 \end{multline*}
In particular $\Omega_{qc}^{2,0}$ is traceless, and $\omega_{qc}$  is the qc connection. In fact, the curvature is contained in the trivial submodules of $R_1$ and $R_3$. This can also be seen as a consequence of the fact that the curvature  is both $G$-invariant and $H$-equivariant as a map
\[\Omega_\theta\colon G\to \Lambda^2T^*\otimes\lie{k},\]
and so must take values in an invariant space.
\end{example}

\begin{example}
Similarly, the qc connection on the homogeneous space \[\Sp(n,1)\Sp(1)/\Sp(n)\Sp(1)\]
is related to the connection $\omega$ considered in Section \ref{sec:example} via
\[\omega_{qc}=\omega+A, \quad A= e^a \otimes w^s\otimes (e_a\hook \omega_s);\]
in this case 
\begin{multline*}
 \Omega_{qc}= \sum_{a<b} e^{ab}\otimes e^a\wedge e^b + \sum_{a<b,s} e^{ab}\otimes e_a\hook\omega_s \wedge e_b\hook\omega_s +(\omega_s+ 2w_s\hook w^{123})\otimes \omega_s\\
-\sum_{a,s,r} e_a\hook\omega_r\wedge w^r \otimes w^s\otimes (e_a\hook \omega_s).
\end{multline*}
Thus, the scalar component in $R_3$ is the same as in the case of the sphere, whilst the component in $R_1$ has the opposite sign.
\end{example}

\begin{example}
Consider now the solvable Lie group $G$  of \cite{ContiFernandezSantisteban:qc} characterized by the existence of left-invariant one-forms $e^1,\dotsc, e^7$ such that
\begin{equation}  
\label{eqn:CFSExample}
\begin{split}
(de^1,\dotsc, de^7)=\bigl(0,e^{15}+ e^{34}- e^{46},  -& e^{24}+ e^{16}+ e^{45},-2   e^{14},\\
e^{12}-e^{34} + e^{46},&e^{13}-e^{42} -  e^{45},e^{14}-e^{23} + e^{56}\bigr)
\end{split}
\end{equation}
The coframe  $e^1,\dotsc, e^7$ defines a $K$-structure. Appling Proposition~\ref{prop:Bintrinsic} to the $(-)$ connection, we obtain:
\[\Theta_{-1}=e^{46}\otimes w_1 -e^{45}\otimes w_2 , \quad g(\Theta_{-1})= e^{134567}\otimes w_1 -e^{124567}\otimes w_2=(D\gamma)^{3,3}.\]
Thus $\Theta_2=\Theta_1=0$ and the structure is integrable.
The qc connection is given by
\[\begin{split}\omega_{qc}=\frac14 e^5\otimes \omega_1+\frac14 e^6\otimes \omega_2-(\frac12 e^4+\frac14e^7)\otimes \omega_3-(\frac32 e^4+\frac12e^7)\otimes (e^{14}+e^{23})&\\
 -\frac34 e^1 \otimes w^s\otimes(e_1\hook\omega_s)   +\frac14 e^2 \otimes w^s\otimes(e_2\hook\omega_s)+\frac14e^3 \otimes w^s\otimes(e_3\hook\omega_s)&\\
 -\frac34e^4 \otimes w^s\otimes(e_4\hook\omega_s).&
  \end{split}
\]
Its curvature is $\Omega_1+\Omega_3$, where
\[\begin{split}
\Omega_1=\frac14&\omega_s\otimes \omega_s+\frac12(e^{14}+e^{23})\otimes \omega_3+\frac12(5e^{14}+e^{23})\otimes (e^{14}+e^{23}),\\
 \Omega_3=\frac18 &e^{67}\otimes \omega_1-\frac18  e^{57}\otimes \omega_2-\frac38e^{56}\otimes \omega_3-\frac12e^{56}\otimes (e^{14}+e^{23})\\
& +\frac1{16}\sum_{a,r,s}( e_a\hook \omega_r)\wedge w^r \otimes w^s\otimes(e_a\hook\omega_s)
 +\frac1{2} e^{47} \otimes w^s\otimes(e_1\hook\omega_s)\\
&+ \frac1{2} (e^{15}- e^{46}) \otimes w^s\otimes(e_2\hook\omega_s)
+\frac1{2}( e^{16}+ e^{45})\otimes w^s\otimes(e_3\hook\omega_s)\\
 &-\frac1{2} e^{17} \otimes w^s\otimes(e_4\hook\omega_s).
\end{split}\]
In terms of the modules of Proposition~\ref{prop:Bianchi}, we see that the curvature is contained in $\tilde R_1+\tilde R_3$, and has a non-zero component in each irreducible submodule of $\tilde R_1+\tilde R_3$.
\end{example}

Recall that two qc manifolds $(M,\mathcal{D})$, $(M',\mathcal{D}')$ are said to be \emph{qc conformal} if there exists a diffeomorphism $\phi\colon M\to M'$ such that $d\phi$ maps $\mathcal{D}$ into $\mathcal{D}'$. It was proved in \cite{IvanovVassilev} that  a qc manifold is locally qc conformal to the Heisenberg group  if and only if a tensor called qc conformal curvature is zero (see also \cite{Kunkel,Alt:Weyl}).

The existence of such a local diffeomorphism that also preserves the $\Sp(n)\h^*$-structure can be characterized by the flatness of the qc connection:
\begin{corollary}
\label{cor:locallyequivalent}
A qc $\Sp(n)\h^*$-structure is locally equivalent to the standard structure on the Heisenberg group if and only if the qc connection is flat.
\end{corollary}
\begin{proof}
On the Heisenberg group,  the $(-)$ connection has torsion $\Theta_0$ and zero curvature, so it satisfies the conditions of Theorem~\ref{thm:hegelianconnection}; consequently, the qc connection is flat, and  one implication is proved.

Conversely, assume that $\omega_{qc}$ is flat. Then every point has a neighbourhood $U$ on which an adapted parallel coframe $e^1,\dotsc, e^{4n+3}$ is defined; since the torsion is $\Theta_0$, it follows that
\[de^1=0 =\dots = de^{4n}, \quad de^{4n+s}=\omega_s, s=1,2,3.\]
In other words, the dual basis of vector fields defines a subalgebra of $\mathfrak{X}(U)$ isomorphic to the Lie algebra of the Heisenberg group; by a standard result (see \cite[Theorem 1.8.3]{Duistermaat}), this shows that $U$ is locally equivalent to the Heisenberg group.
\end{proof}
Since  the qc connection on the sphere is not flat, it follows that the standard  $\Sp(n)\h^*$-structures on the sphere and the Heisenberg group are not equivalent. However, it was shown in \cite{IvanovMinchevVassilev} that these qc manifolds are locally conformally equivalent via the Cayley transform; this is not an equivalence of $\Sp(n)\h^*$-structures, because it does not preserve the complement.

\section{Qcm structures}
\label{sec:posit}
There is no natural  choice of metric on a qc $K$-structure, although there is a conformal class of metrics on the horizontal distribution. In this section we begin by fixing a metric in this class, partly to compare  our results with those of \cite{Biquard,Duchemin}. Indeed, we recover the known fact that the choice of metric determines the complement in full. In addition we construct  a canonical metric connection whose torsion lies in an affine space parallel to $S^2(EH)+W\otimes EH$. Its curvature lies in a ``small'' submodule which we identify, and is almost entirely determined by the torsion and its derivative. This connection differs from the Biquard connection; it should be regarded as the metric version of the qc connection of Section~\ref{sec:curvature}.

In our language, the choice of a metric amounts to considering an arbitrary reduction of an integrable qc $K$-structure to $G=\Sp(n)\Sp(1)\ltimes EH$. Since $\partial_G$ is injective, the argument of the remark on p.~\pageref{remark:snake} shows that the sequence
\[0\to \ker\partial_K\to T^*\otimes \R\xrightarrow{\partial} \coker \partial_{G}\to \coker\partial_K\to 0.\]
is exact.  Thus, the $G$-intrinsic torsion takes values in $\Theta_0^G+\partial(T^*\otimes \R)$. In analogy with Theorem~\ref{thm:canonicalKreduction}, this suggests that some reduction to $\Sp(n)\Sp(1)$  kills a component of the intrinsic torsion isomorphic to $EH$.

It turns out that such a reduction exists and is unique; in other words, a qc distribution and a compatible metric on the distribution determine a canonical  $\Sp(n)\Sp(1)$-structure:
\begin{proposition}
\label{prop:metricdeterminescomplement}
Given an integrable  qc structure and a choice of metric on the horizontal distribution compatible with the structure, there is a unique reduction $P$ to $\Sp(n)\Sp(1)$ such that $\Theta^G|_P\equiv \Theta_0^G$ and the metric is compatible with $P$.
\end{proposition}
\begin{proof}
Corollary~\ref{cor:Kreduction} gives a canonical $K$-reduction; by Theorem~\ref{thm:hegelianconnection}, its $K$-intrinsic torsion is $\Theta_0^K$. The choice of metric gives a reduction to $G$.

In terms of the structure group $\Sp(n)\Sp(1)$, the calculations of Lemma~\ref{lemma:Kintrinsictorsion} give
\[
( w^s\otimes v\hook\omega_s) \omega_r\otimes w_r-\Theta_0 = ( w^s\otimes v\hook\omega_s) \cdot\Theta_0\mod\im\partial_{G}.\]
By Lemma~\ref{lemma:linearized}, the right hand side lies in the $EH$ containing 
$\frac12\tilde\alpha_1-\alpha_4+i\alpha_5$,
which is not in $\im\partial_{G}$. Thus,  the argument of Theorem~\ref{thm:canonicalKreduction} applies, and we find a unique reduction to $\Sp(n)\Sp(1)$ that satisfies the required condition.

The uniqueness follows from the fact that the torsion condition in the definition implies that the $B$-intrinsic torsion is $\Theta_0^B$, which makes the reduction to $K$ unique.
\end{proof}

We can now prove that the two definitions of integrability for seven-dimensional qc structures agree:
\begin{corollary}
\label{cor:integrableisintegrable}
A qc structure is integrable in the sense of Duchemin if and only if it is compatible with an integrable qc $K$-structure.
\end{corollary}
\begin{proof}
An integrable qc structure in the sense of Duchemin is a qc $B$-structure that about each point admits a section $s$ such that (with obvious notation)
\[(s^*d\eta)^{1,1}\in \Lambda^{1,1}\otimes W\cong V\otimes W\otimes W\]
is skew-symmetric in the last two indices.
Fix a $B$-connection $\omega$ on such a structure; Proposition~\ref{prop:Bintrinsic} gives
\[(d\eta+\omega\wedge\eta)^{1,1}=(D\eta)^{1,1} = \Theta_2+\Theta_{-1}.\]
By construction, both $s^*(\omega\wedge\eta)$ and $\Theta_{-1}$ have no component in $V\otimes S^2_0W$. Thus, if $s^*(d\eta)^{1,1}$ is skew-symmetric then $\Theta_2$ is zero. 

Conversely, given an integrable qc $K$-structure, take  a reduction $P$ such as in Proposition~\ref{prop:metricdeterminescomplement}, and choose a $G$-connection on $P$ with torsion
$\Theta=\Theta_0$. Then 
\[(d\eta+\omega\wedge\eta)^{1,1}=(D\eta)^{1,1} = 0;\]
because of how $G$ is defined,  in this case $s^*(\omega\wedge\eta)$ has no component in $V\otimes S^2W$, so $s^*(d\eta)^{1,1}$ is skew-symmetric for any section of  $P$.
\end{proof}

Proposition~\ref{prop:metricdeterminescomplement}  motivates the following:
\begin{definition}
\label{def:qcm}
A quaternionic-contact metric (qcm) structure  on a $4n+3$-dimensional manifold is 
an  $\Sp(n)\Sp(1)$-structure with $G$-intrinsic torsion equal to $\Theta_0^G$.
\end{definition}
In this language, Proposition~\ref{prop:metricdeterminescomplement} asserts that an integrable qc structure determines a family of qcm structures, one for each choice of compatible metric on the horizontal distribution. Notice that the qc structure underlying a qcm structure is always integrable.

\smallskip
Having reduced the structure group, we refine the usual decomposition as
\[\begin{split}
\Lambda^2T^*\otimes T = \im(\partial_{\Sp(n)\Sp(1)})\oplus\Lambda^2V^*\otimes W  &\oplus \partial_1(W_1)\oplus\partial_2(W_2)\oplus \partial(W_3)\\ &\oplus\partial(V^*\otimes\R)\oplus \partial(T^*\otimes EH )
  \end{split},
\]
and decompose the torsion of any $\Sp(n)\Sp(1)$-connection as
\begin{equation}
 \label{eqn:SpnSp1decomp}
\Theta=\Theta_*+\Theta^Q+\Theta_1+\dots + \Theta_5.
\end{equation}
By construction, any connection on a qcm structure  satisfies
\begin{equation*}
\Theta^Q=\Theta_0, \quad\Theta_1=\dots = \Theta_4=0.
\end{equation*}

We shall need the following lemma in order to characterize the curvature of the qcm connection. The final part of the lemma will also be used in the proof of Corollary~\ref{cor:fourformclosed}.
\begin{lemma}
\label{lemma:projections}
Consider the projections
\[p_1^{\lie{sp}(1)}\colon \tilde R_1\to \Lambda^{2,0}\otimes\lie{sp}(1), \quad p_2^{\lie{sp}(1)}\colon \tilde R_2\to \Lambda^{1,1}\otimes\lie{sp}(1),\]
\[p_2^{EH}\colon \tilde R_2\to \Lambda^{2,0}\otimes EH,\quad  p_3^{EH}\colon \tilde R_3\to \Lambda^{1,1}\otimes EH.\]
Then 
\[\im p_1^{\lie{sp}(1)}\cong S^2(EH), \quad \im p_2^{\lie{sp}(1)}\cong EH+ES^3H, \quad \ker p_2^{EH}=0= \ker  p_3^{EH}.\]
Moreover, if $n>1$, and $\R$ denotes the trivial module in $\Lambda^{1,0}\otimes EH$,
\[\im p_2^{EH}\cap (\lie{sp}(1)\otimes EH + \Lambda^{1,0}\wedge\R) = 0.\]
If $n=1$, $\im p_2^{EH}$ is the direct sum of $ES^3H\subset \lie{sp}(1)\otimes EH$ and a diagonal $EH$  in $\lie{sp}(1)\otimes EH + \Lambda^{1,0}\wedge\R$.
\end{lemma}
 \begin{proof}
By construction $\ker p_1^{\lie{sp}(1)}$ is the preimage  of $s(V^*\otimes\widetilde{EH})$ in $\Lambda^{2,0}\otimes \lie{sp}(n)$ under $\delta$.  It is easy to check that the submodule of $s(V^*\otimes \widetilde{EH})$ isomorphic to $S^2(EH)$ is transverse to $\Lambda^{3,0}\otimes V$, which contains $\delta(\Lambda^{2,0}\otimes \lie{sp}(n))$. Thus, $\ker p_1^{\lie{sp}(1)}$ is the kernel of  $\delta$ in $\Lambda^{2,0}\otimes \lie{sp}(n)$.

The fact that the image of $p_2^{\lie{sp}(1)}$ contains  $EH+ES^3H$ follows immediately from \eqref{eqn:ES3HinR2} and \eqref{eqn:EHinR2}. By Schur's lemma, equality holds.

The projection $p_3^{EH}$ is injective because otherwise $\delta$ would not be injective on $\Lambda^{0,2}\otimes ( \lie{sp}(n)+\lie{sp}(1))$, which is absurd.

Similarly, if $p_2^{EH}$ were not injective then $\delta(\Lambda^{1,1}\otimes(\lie{sp}(n)+\lie{sp}(1)))$ would intersect  $s(W^*\otimes\widetilde{EH})$ non-trivially; then $\partial(\Lambda^{1,0}\otimes 
(\lie{sp}(n)+\lie{sp}(1)))$ would intersect $\widetilde EH$, which is absurd because by Lemma~\ref{lemma:formulae} $\tilde\alpha_1$ is not a linear combination of $\partial(\alpha_1)$ and $\partial(\alpha_2)$.

The last part of the statement amounts to proving that the intersection
\[\delta(\Lambda^{2,0}\otimes (\lie{sp}(n)+\lie{sp}(1)) + \lie{sp}(1)\otimes EH+ \Lambda^{1,0}\wedge\R)\cap s(W^*\otimes \widetilde{EH})\]
is  $EH+ES^3H$ when $n=1$ and zero otherwise.
Computing with highest weight vectors, we see that for $n=1$
\begin{multline*}
 s((w^2+iw^3)\otimes ( \frac12\tilde\alpha_1-\alpha_4+i\alpha_5)\\
=-\frac12\delta\biggl(((w^2+iw^3)\wedge v_1h_1+ i w^1\wedge v_1h_2)\otimes (\omega_2+i\omega_3)
+(w^2+iw^3)\wedge\alpha_2\\
+v_jh_2\wedge v_{n+j}h_2 \wedge \bigl(iw^1\otimes v_1h_2 +  (w^2+iw^3)\otimes v_1h_1\bigr)\biggr),
\end{multline*}
showing that the intersection contains $ES^3H$; no such equality holds  for $n>1$. Similarly, one verifies that the intersection only contains $EH$ when $n=1$, but the relevant $EH$ projects non-trivially to both  $\lie{sp}(1)\otimes EH$ and $\Lambda^{1,0}\wedge\R$.
\end{proof}

\begin{theorem}
\label{thm:qcmintrinsic}
The intrinsic torsion of a qcm structure  lies in 
\[\Span{\Theta_0} + S^2ES^2H+\Lambda^2_0E+\R+EH+ES^3H.\]
In particular any qcm structure has a unique connection with torsion
\[\Theta_0+\partial(\chi_V)+\partial(\chi_W), \quad \chi_V\in S^2(EH)\subset V^*\otimes EH, \quad \chi_W\in W^*\otimes EH;\]
its curvature satisfies
\begin{gather*}
\Omega^{2,0}\in S^2(EH) + S^4E, \quad 
\Omega^{1,1}\in S^3EH+ES^3H+EH,\\
\Omega^{0,2}\in S^2ES^2H+S^4H+\R.
\end{gather*}
Moreover there are linear equivariant maps
\begin{gather*}
f_1\colon S^2(EH)\to \Lambda^{2,0}\otimes (\lie{sp}(n)+\lie{sp}(1)), \quad
f_2\colon W^*\otimes EH\to \Lambda^{1,1}\otimes (\lie{sp}(n)+\lie{sp}(1))\\
f_3\colon \Lambda^{2,0}\otimes EH\to  \Lambda^{1,1}\otimes (\lie{sp}(n)+\lie{sp}(1)), \quad
f_4\colon \Lambda^{1,1}\otimes EH\to  \Lambda^{0,2}\otimes (\lie{sp}(n)+\lie{sp}(1))
\end{gather*}
such that $f_1$ and $f_2$ are injective,
\begin{gather*}
\Omega^{2,0}-f_1(\chi_V) \in S^4E, \quad \Omega^{1,1}-f_2(\chi_W)\in S^3EH, 
\quad \Omega^{0,2}=f_4(D(\chi_V+\chi_W)^{1,1}),
\end{gather*}
and $\Omega^{1,1}-f_3((D\chi_V)^{2,0})$ is in  $ES^3H$ (and zero when $n>1$).
\end{theorem}
\begin{proof}
Consider an $\Sp(n)\Sp(1)$ connection $\omega$ with torsion $\Theta_0+\Theta_5$, where 
\[\Theta_5=\partial( \chi_V)+\partial(\chi_W), \quad \chi_V\in V^*\otimes EH, \quad \chi_W\in W^*\otimes EH.\]
Then 
\[\omega_{qc}=\omega-\chi_V-\chi_W\]
is a $K$-connection with curvature
\[\Omega_{qc}=\Omega - D(\chi_V+\chi_W);\]
in particular,  $\tr\Omega_C$  is zero, so $\omega_{qc}$ is the qc connection and, in the notation of Lemma~\ref{lemma:Kcurvature},
\[\Omega_{qc}\in \tilde R_1+\tilde R_2+\tilde R_3+R_4.\]
In particular
\begin{equation*}
[\Omega_{qc}^{2,0}]_{\lie{sp}(1)} \in S^2ES^2H+\Lambda^2_0E+\R, \quad [\Omega_{qc}^{1,1}]_{\lie{sp}(1)}\in ES^3H+EH.
\end{equation*}
Since $\eta$ and $\Theta_0$ are parallel under $\omega$, we compute
\[D\eta = \Theta_0\hook\eta=\Theta_0, \quad \Omega\wedge\eta = D\Theta_0= \Theta_5\hook\Theta_0.\]
Now
\[(\Theta_5\hook \Theta_0)^{2,1}= \partial (\chi_V)\hook\Theta_0, \quad (\Theta_5\hook \Theta_0)^{1,2}= \partial (\chi_W)\hook\Theta_0.\]
Notice that applying this argument to the $(0,3)$ part we obtain no information, as both $ [\Omega_{qc}^{0,2}]_{\lie{sp}(1)}\wedge\eta$ and  $(\Theta_5\hook\Theta_0)^{0,3}$ are zero.

Using the fact that  $\Omega = [\Omega_{qc}]_{\lie{sp}(n)+\lie{sp}(1)}$, we deduce
\begin{equation}
\label{eqn:chiDetermineOmega}
[\Omega_{qc}^{2,0}]_{\lie{sp}(1)} \wedge \eta = \partial( \chi_V)\hook\Theta_0, \quad [\Omega_{qc}^{1,1}]_{\lie{sp}(1)} \wedge \eta  = \partial (\chi_W)\hook\Theta_0.
\end{equation}
It is easy to verify that the map
\begin{equation}
\label{eqn:TotimesEHskew}
T^*\otimes EH\to \Lambda^{2,1}\otimes W, \quad v\to \partial (v)\hook\Theta_0
\end{equation}
is injective; thus, \eqref{eqn:chiDetermineOmega} implies that $\chi_V$ is in $S^2(EH)$. 
By Lemma~\ref{lemma:projections}, $\chi_V$ and $\chi_W$ determine part of the curvature, and the dependence can be expressed by  linear maps $f_1$, $f_2$ as in the statement. Notice that $f_1$ and $f_2$ are necessarily injective because so is the map \eqref{eqn:TotimesEHskew}.

We can now write
\[D\chi_V \in T^*\wedge S^2(EH) +\Theta\hook \chi_V, \quad D\chi_W\in (\Lambda^{1,1}+\Lambda^{0,2})\otimes EH+\Theta\hook\chi_W,\]
where
\[\Theta\hook\chi_V=\partial(\chi_V+\chi_W)\hook\chi_V, \quad \Theta\hook\chi_W=\Theta_0\hook\chi_W;\]
therefore,
\begin{gather*}
\Omega^{2,0}_{qc}=\Omega^{2,0}-	\Theta_0\hook\chi_W-(D\chi_V)^{2,0}, \quad \Omega^{1,1}_{qc}=\Omega^{1,1}-(D\chi_V)^{1,1}-(D\chi_W)^{1,1},\\
\Omega^{0,2}_{qc}=\Omega^{0,2}-\partial(\chi_W)\hook \chi_V-(D\chi_W)^{0,2}. 
\end{gather*}
In consequence,
\begin{gather*}
\Omega^{2,0}\in \tilde R_1, \quad \Omega^{1,1}-\Theta_0\hook\chi_W-(D\chi_V)^{2,0}\in\tilde R_2,\\
\Omega^{0,2}-(D\chi_V)^{1,1}-(D\chi_W)^{1,1}\in \tilde R_3.
\end{gather*}
Taking $p_2^{EH}$ we see that $\Omega^{1,1}$ is determined by $\Theta_0\hook\chi_W+(D\chi_V)^{2,0}$, where the first summand lies in $S^2H\otimes EH$. It also follows from  Lemma~\ref{lemma:projections} that 
 $(D\chi_V)^{2,0}$ alone determines all of $\Omega^{1,1}$ when $n>1$, and its component $S^3EH+EH$ when $n=1$.

Similarly, the fact that $p_3^{EH}$ is injective shows that $(D\chi_V+D\chi_W)^{1,1}$ determines $\Omega^{0,2}$.
\end{proof}
We shall refer to the connection of Theorem~\ref{thm:qcmintrinsic} as the qcm connection; it can be regarded as a canonical object of qcm geometry in the same way that the qc connection is canonical in qc geometry,  since the two are related by  the obvious projection from  $\lie{k}$ to $\lie{sp}(n)+\lie{sp}(1)$. The qcm connection is also natural in that it coincides with the natural connection that appears in the three fundamental examples of Section~\ref{sec:example}. 

\begin{remark}
The curvature of the qcm connection has the remarkable property that the torsion (together with its covariant derivative) determines all of its curvature except for $S^4E\subset\tilde R_1$. This component of the curvature is the same whether one considers the qc, qcm or Biquard connection (see Corollary~\ref{cor:Omegabiq}), and can be identified with the qc conformal curvature tensor constructed in \cite{IvanovVassilev}.
\end{remark}

\begin{remark}
We could have considered the curvature of the qc connection instead, but this would have changed little, because its curvature, as shown in the course of the proof,  can be identified  with the curvature of the qcm connection by means of a  projection to $\lie{sp}(n)\lie{sp}(1)$. The fact that the modules appearing in Theorem~\ref{thm:qcmintrinsic}  are smaller than those of Proposition~\ref{prop:Bianchi} is a consequence of the fact that we are now working with a qcm structure rather than an arbitrary $\Sp(n)\Sp(1)$ reduction of a qc structure.
\end{remark}

Recall that the (horizontal) Ricci tensor is defined as
\[\ric(X,Y)=\sum_{a=1}^{4n} \langle e^a, \Omega(e_a,X)Y\rangle, \quad X,Y\in V;\]
notice that the contraction is only taken on indices along $V$. Its trace is called the qc scalar curvature. The  Ricci tensor arises from an equivariant map $\Lambda^{2,0}\otimes(\lie{sp}(n)+\lie{sp}(1))\to V^*\otimes V^*$. Restricting to $\tilde R_1$, it follows by Schur's lemma that the image is contained in  $S^2(V^*)$: so, the Ricci tensor is symmetric, as proved in  \cite{Biquard} for the Biquard connection. 
\begin{remark}
With a bit of work one can show that the image is exactly $S^2(V^*)$, so the Ricci of the qcm connection can be identified with $\chi_V$. In particular, the qc scalar curvature is the scalar part of $\chi_V$. The remaining part of the intrinsic torsion, $\chi_W$, is easily seen to be the obstruction to the complement distribution being integrable.
\end{remark}

A qc manifold is called  \dfn{qc-Einstein} if the Biquard connection has zero traceless Ricci. This is a very strong condition, due to results of \cite{IvanovVassilev:ClosedFundamentalForm}.  We can now prove an analogous result which  uses the qcm connection instead.
\begin{corollary}
\label{cor:fourformclosed}
On a qcm manifold of dimension greater than seven, the following conditions are equivalent:
\begin{enumerate}
\item the four-form $\sum_s\omega_s^2$ is closed;
\item $\chi_V$ is a constant scalar and $\chi_W=0$;
\item the qcm connection has curvature in $S^4E+2\R\subset R_1+ R_3$; 
\item the horizontal traceless Ricci of the qcm connection is zero;
\item the traceless part of $\chi_V$ is zero.
\end{enumerate}
\end{corollary}
\begin{proof}
The fundamental form $\sum_s\omega_s^2$ is parallel under $\omega_{qcm}$; therefore, it is closed if $\Theta\hook \sum \omega_s^2=0$, or 
\[\partial(\chi_V)\hook \sum_s\omega_s^2=0, \quad\partial(\chi_W)\hook \sum_s\omega_s^2=0.\]
The second equation is equivalent to  $\chi_W=0$, because interior product with the fundamental form is an injection of $V$ into $\Lambda^{3,0}$. The first equation means that $\chi_V$ is in the kernel of a map $S^2(EH)\to \Lambda^{4,1}$ which is easily checked to have kernel $\R$ (since we assume $n>1$).
Thus, (1) is equivalent to (2). 

The fact that (2) implies (3) follows from Theorem~\ref{thm:qcmintrinsic}; on the other hand, (3) obviously implies (4).

The fact that (4) is equivalent to (5) follows from 
\[\Omega^{2,0}-f_1(\chi_V) \in S^4E,\]
and the fact that the Ricci contraction from $\tilde R_1$ to $S^2(EH)$ has kernel $S^4E$, as noted in a remark above.

Now assume (5) holds; then $(D\chi_V)^{2,0}\in V^*\wedge\R$. By Lemma~\ref{lemma:projections},
\[ p_2^{EH}(\Omega^{1,1}-\Theta_0\hook\chi_W-(D\chi_V)^{2,0})=0.\]
This implies that both $\chi_W$ and $(D\chi_V)^{2,0}$ are zero.

Similarly, $(D\chi_V)^{1,1}$ lies in as module isomorphic to $S^2H$, so 
\[\Omega^{0,2}-(D\chi_V)^{1,1}-(D\chi_W)^{1,1}\in \tilde R_3\]
implies that $(D\chi_V)^{1,1}$ is zero. Summing up, $\nabla\chi_V$ is zero, i.e. $\chi_V$ is constant. This establishes the equivalence of (2) and (5), completing the proof.
\end{proof}
\begin{remark}
One can rephrase Corollary~\ref{cor:fourformclosed} in terms of the Biquard connection, for condition (5) is equivalent to $T_\xi=0$ (Corollary~\ref{cor:TxiisChiV}), and condition (4) is equivalent to qc-Einstein (Corollary \ref{cor:einsteiniseinstein}). This version of the  statement  was proved in \cite{IvanovVassilev:ClosedFundamentalForm}.
\end{remark}

\begin{remark}
There are two points in the proof where the assumption on the dimension is used. First, the  map $S^2(EH)\to \Lambda^{4,1}$ whose kernel contains $\chi_V$ is zero when $n=1$, so (1) does not imply (2) for $n=1$. In fact, this implication has a known counterexample \cite{ContiFernandezSantisteban:qc}.

Secondly, the fact that $\chi_V$ is a scalar does not apparently force $\chi_W$ to vanish for $n=1$, because of the form that Lemma~\ref{lemma:projections} takes in this instance. A very recent result \cite{IvanovMinchevVassilev:qcEinstein} shows however that even when the dimension is seven, $\chi_V$ is a scalar only when the vertical distribution is integrable, hence $\chi_W=0$.
\end{remark}

We conclude this section by noting that the natural qcm structure in each example of Section~\ref{sec:example} satisfies the conditions of Corollary~\ref{cor:fourformclosed}, with $\chi_V$ a positive, negative or zero constant. In these examples, the $S^4E$ component of the qcm curvature is also zero. More generally, 3-Sasakian manifolds have a natural qc-Einstein structure; the converse also holds up to local homothety (see \cite{IvanovVassilev:ClosedFundamentalForm}).

\section{The Biquard connection}
\label{sec:biquard}
In this section we compare our results with those of \cite{Biquard,Duchemin}. In fact, we recover the existence of the Biquard and Duchemin connection, and express them in terms of the qcm connection, showing  that all three connection  exist in all dimensions. We show that the component the torsion of the Biquard connection usually denoted by $T_\xi$  can be identified with a traceless symmetric endomorphism of $\mathcal{D}$, which can be identified with the traceless Ricci of either the Biquard or the qcm connection.

We shall  decompose any $\eta\in\Lambda^2T^*\otimes T$ as
\[[\eta]_V+[\eta]_W, \quad [\eta]_V\in \Lambda^2T^*\otimes V, \quad [\eta]_W=\Lambda^2T^*\otimes W.\]

Recall the following:

\begin{theorem}[Biquard  \cite{Biquard}]
If $n>1$, given a  qc structure with a fixed compatible metric on the distribution $\mathcal{D}$, there is a unique complement $W_{B}$ and a unique connection which:
\begin{enumerate}
\item preserves both $\mathcal{D}$ and $W_{B}$, as well as the $\Sp(n)\Sp(1)$-structure on $\mathcal{D}$, and acts on $W_B$ as on the subbundle of $\End(V)$ determined by the almost complex structures;
\item satisfies the torsion conditions 
\begin{equation}
\label{eqn:Biquard}
[\Theta^{2,0}]_V=0,\quad [\Theta^{1,1}]_V\in \partial (W^*\otimes (\lie{sp}(n)\lie{sp}(1))^\perp),
 \end{equation}
\end{enumerate}
where the  orthogonal complement is taken in $\lie{gl}(V)$.
\end{theorem}

In seven dimensions, we have the following similar result:
\begin{theorem}[Duchemin \cite{Duchemin}]
If $n=1$,  given a qc structure with a fixed compatible metric on the distribution $\mathcal{D}$, then there is a unique complement $W_{D}$ and a unique connection which:
\begin{enumerate}
\item preserves both $\mathcal{D}$ and $W_{D}$, and the metrics on them;
\item satisfies the torsion conditions
\begin{equation}
 \label{eqn:Duchemin}
[\Theta^{2,0}]_V=0, \quad [\Theta^{0,2}]_W=0, \quad \Theta^{1,1}\in ES^5H+W^*\otimes S^2V.
\end{equation}
The $ES^5H$ component is zero if and only if the qc structure is integrable.
\end{enumerate}
\end{theorem}
We shall refer to these connections as the Biquard connection and the Duchemin connection. Notice that the Duchemin connection has holonomy contained in $\SO(4)\times \SO(3)$, or $\Sp(n)\Sp(1)\times\SO(W)$. 

We introduce two equivariant maps
 \[T_B\colon S^2(EH)\to W^*\otimes (\lie{sp}(n)+\lie{sp}(1)), \quad T_D\colon S^2(EH)\to W^*\otimes \so(W) ;\]
equivariance implies that $T_B$ is zero on $\Lambda^2_0E$ and $T_D$ is zero on  $\Lambda^2_0E+S^2ES^2H$: on the remaining components, we set
\begin{gather*}
T_D\bigl( e_a\otimes w^s\otimes e_a\hook\omega_s\bigr)=4w^s\otimes (w_s\hook w^{123}), \quad T_B\bigl( e_a\otimes w^s\otimes e_a\hook\omega_s\bigr)=2w^s\otimes\omega_s,\\
T_B(v_1h_2\otimes((w^2+iw^3)\otimes v_1h_1 + i w^1\otimes v_1h_2))= \frac12 (w^2+iw^3)\otimes v_1h_2\wedge v_1h_1.
\end{gather*}

We can recover the existence of the Biquard and Duchemin connection as a consequence of what we have proved so far; in particular, we show that our choice of complement coincides with $W_B$ and $W_D$. \begin{theorem}
\label{thm:biquard}
On a qcm structure:
\begin{itemize}
 \item  there is a unique connection $\omega_B$ whose torsion $\Theta_B$ satisfies \eqref{eqn:Biquard}; it is related to the qcm connection via
\[\omega_B=\omega_{qcm}+T_B(\chi_V),\]
and satisfies
\[\Theta_B^{2,0}=\Theta_0, \quad \Theta_B^{1,1} \in  S^2ES^2H+\Lambda^2_0E, \quad \Theta_B^{0,2}\in EH+ES^3H+\R;\]
\item there is a unique  $\Sp(n)\Sp(1)\times\SO(W)$-connection $\omega_D$  whose torsion $\Theta_D$ satisfies \eqref{eqn:Duchemin}, given by
\[\omega_D=\omega_{qcm}+T_B(\chi_V)+T_D(\chi_V);\]
moreover $\Theta_D=\Theta_B-[\Theta_B^{0,2}]_W$.
\end{itemize}

Given a qc structure and a metric on the horizontal distribution:
\begin{itemize}
 \item  there is a unique $\Sp(n)\Sp(1)$-structure compatible with structure and metric that admits a connection satisfying \eqref{eqn:Biquard};
 \item  there is a unique $\Sp(n)\Sp(1)$-structure compatible with structure and metric that admits an $\Sp(n)\Sp(1)\times\SO(W)$-connection satisfying \eqref{eqn:Duchemin}.
\end{itemize}
These two $\Sp(n)\Sp(1)$-structures coincide, they are a qcm structure $P$ in the sense of Definition~\ref{def:qcm}, and the natural complement 
\[P\times_{\Sp(n)\Sp(1)} W\subset P\times_{\Sp(n)\Sp(1)} T\]
  coincides with the complement $W_B$ (resp. $W_D$).
\end{theorem}
\begin{proof}
Let $\omega_{qcm}$ be the qcm connection; its torsion is determined by $\chi_V$, $\chi_W$ as defined in Theorem~\ref{thm:qcmintrinsic}. Consider the connection
\[\omega_B=\omega_{qcm}+\eta_V+\eta_W, \quad \eta_V\in V^*\otimes(\lie{sp}(n)+\lie{sp}(1)), \quad \eta_W\in W^*\otimes(\lie{sp}(n)+\lie{sp}(1));\]
its torsion $\Theta_B$ satisfies 
\begin{gather*}
(\Theta_B)^{2,0}=\Theta_0 + \partial(\eta_V)^{2,0} , \quad (\Theta_B)^{0,2}=\partial(\chi_W)+\partial(\eta_W)^{0,2},\\
(\Theta_B)^{1,1}=\partial(\chi_V)+\partial(\eta_V)^{1,1}+\partial(\eta_W)^{1,1}.
\end{gather*}

The first equation in \eqref{eqn:Biquard} is equivalent to $\eta_V=0$, whereas the second equation means that
\[\partial(\chi_V)+\partial(\eta_W)^{1,1} \in  \partial(W^*\otimes(\lie{sp}(n)+\lie{sp}(1))^\perp).\]
There is a unique solution in $\eta_W$, proving existence and uniqueness of $\nabla_B$. Since $T_B$ satisfies
\[\partial( T_B(v)) + \partial (v) \in \partial(W^*\otimes(\lie{sp}(n)+\lie{sp}(1))^\perp), \quad v\in S^2(EH),\]
the solution is given by $\eta_W=T_B(\chi_V)$. Notice that $\partial \circ T_B+\partial$ is zero on the component isomorphic to $\R$, so $(\Theta_B)^{1,1}$ lies in $S^2ES^2H+\Lambda^2_0E$.

To determine  $\omega_D$, we think of the Lie algebra of  $\Sp(n)\Sp(1)\times\SO(W)$ as  $(\lie{sp}(n)+\lie{sp}(1)) + \so(W)$, with $\lie{sp}(1)$ contained diagonally in $\so(V)+\so(W)$ as usual. Accordingly, 
we can write an arbitrary   $\Sp(n)\Sp(1)\times\SO(W)$-connection as $\omega_D=\omega_{qcm}+\eta_V+\eta_W+ A$, where $\eta_V$ and $\eta_W$ are as in the first part of the proof and  $A$ is an $\so(W)$-valued tensorial $1$-form. The torsion is then
\[\begin{split}
(\Theta_D)^{2,0}&=\Theta_0 + \partial(\eta_V)^{2,0} ,\quad (\Theta_D)^{0,2}=\partial(\chi_W)+\partial(\eta_W)^{0,2}+\partial (A)^{0,2},\\
(\Theta_D)^{1,1}&= \partial(\chi_V)+\partial(\eta_V)^{1,1}+\partial(\eta_W)^{1,1} +\partial (A)^{1,1}.
   \end{split}
\]
The condition on $(\Theta_D)^{2,0}$ immediately implies that $\eta_V=0$; so, the condition on $(\Theta_D)^{1,1}$ reads
\[\partial (A)^{1,1}=0, \quad  \partial(\chi_V)+\partial(\eta_W)^{1,1} \in W^*\otimes (\lie{sp}(n)+\lie{sp}(1))^\perp.\]
Therefore $\eta_W=T_B(\chi_V)$. Imposing
\[0=[(\Theta_D)^{0,2}]_W=\partial(T_B(\chi_V))^{0,2}+(\partial A)^{0,2}\]
gives $A=T_D(\chi_V)$. It is now clear that the torsion $\Theta_D$ differs from $\Theta_B$ only in that the $\R$ component is zero.

In order to prove the uniqueness of the qcm structure, observe that a qc structure and a metric determine a qc $\Sp(n)\Sp(1)\ltimes \Hom(W,V)$-structure. Take an arbitrary reduction to $\Sp(n)\Sp(1)$, and assume it has a connection satisfying \eqref{eqn:Biquard}. Decomposing its torsion according to \eqref{eqn:SpnSp1decomp}, we find
\[[\Theta^{2,0}]_V=[(\Theta_*)^{2,0}]_V+[(\Theta_1)^{2,0}]_V+[(\Theta_4)^{2,0}]_V.\]
 Write 
\[
 \Theta_*=\partial(\eta_V+\eta_W), \quad 
\Theta_1=\partial_1(\eta_1), \quad \Theta_4=\partial(\delta_V+\delta_W),\]
with obvious notation; then
$(\eta_V,\eta_1,\delta_V)$ is in the kernel of 
\[ V^*\otimes(\lie{sp}(n)+\lie{sp}(1))\oplus W_1 \oplus V^*\otimes\R\to \Lambda^2V^*\otimes V, \quad v\to [\partial (v)^{2,0}]|_V.\]
This map is surjective with kernel  $2EH+ES^3H$. To identify these subspaces more precisely,  observe that $V^*\otimes(\lie{sp}(n)+\lie{sp}(1))$ intersects $W_1$ in $V^*\otimes\lie{sp}(1)$; thus, the kernel contains $(\alpha_1,-\alpha_1,0)$ and $(\beta_1,-\beta_1,0)$. The calculations  of Lemma~\ref{lemma:formulae} show that the kernel also contains
\[(\alpha_2,\alpha_3,v_1h_2\otimes(4\id_V+8\id_W));\]
in particular, the kernel projects injectively on $W_1$.

By  Theorem~\ref{thm:reductiontheorem}, $\eta_1$ lies in the  $ES^3H$ containing $2\beta_2+3\beta_1$, which intersects  $V^*\otimes\lie{sp}(1)$ trivially; hence, $\eta_1=0$. Thus, $\Theta_1=0$; by Theorem~\ref{thm:reductiontheorem} and integrability, this implies that $\Theta_2=0$, so $\Theta^B=\Theta_0^B$. In addition, $\delta_V$ is also forced to be zero, so $P$ is a qcm structure.
\end{proof}
This result shows that both the Duchemin and Biquard connections exist in all dimensions; moreover, they only differ by a component in $\R\subset W^*\otimes\so(W)$, which has the effect of killing $[\Theta^{0,2}]_W$.

The tensor $T_\xi$ that appears in the literature can be identified with  $\Theta^{1,1}$. It is customary to decompose $T_\xi$ as 
\[T_\xi=T^0_\xi + b_\xi, \quad b_\xi \in W^*\otimes \so(V), \quad T^0_\xi \in W^*\otimes S^2V.\]
Decomposing into irreducible modules,
\[b_\xi\in \Lambda^2_0E(S^4H+S^2H+\R), \quad T^0_\xi\in S^2E(S^4H+S^2H+\R) + \Lambda^2_0ES^2H + S^2H;\]
however, it follows from Theorem~\ref{thm:biquard} that $b_\xi$ is really contained in $\Lambda^2_0E$ and $T^0_\xi$ in $S^2ES^2H$. This is consistent with the results of \cite{Biquard} (see also  \cite[Proposition 2.4]{IvanovMinchevVassilev} and \cite[Proposition 3.1]{Duchemin:hypersurfaces}).

As an element of $S^2ES^2H+\Lambda^2_0E$, $T_\xi$ can be viewed as a traceless symmetric tensor; in fact, it can  be identified with the traceless part of $\chi_V$:
\begin{corollary}
\label{cor:TxiisChiV}
There is a linear map $S^2(EH)\to W^*\otimes (\lie{sp}(n)+\lie{sp}(1))^\perp$ with kernel $\R$ that maps $\chi_V$ to $T_\xi$.
\end{corollary}
\begin{proof}
It suffices to check that the map 
\[S^2(EH)\to \Lambda^2T^*\otimes T , \quad \chi_V \to \partial(\chi_V)+T_B(\chi_V)\]
has kernel $\R$.
\end{proof}

\begin{example}
Theorem~\ref{thm:biquard} also tells us how to compute the torsion of the Biquard connection from that of the qcm connection. For instance, the qcm connection  on the Lie group \eqref{eqn:CFSExample} satisfies
\[\chi_V=\frac12 (e^{1}\otimes e^1 +e^4\otimes e^4-e^2\otimes e^2-e^3\otimes e^3)+\frac14e^a\otimes e_a, \quad \chi_W=0;\]
therefore the $S^2ES^2H$ component of torsion of the Biquard connection is the projection to $W\otimes S^2V\subset \Lambda^{1,1}\otimes V$ of 
\[\partial(\frac12 (e^{1}\otimes e^1 +e^4\otimes e^4-e^2\otimes e^2-e^3\otimes e^3))\]
i.e.
\[(\Theta_B)^{1,1}=-\frac12 e^5\wedge (e^1\odot e^2-e^3\odot e^4) -\frac12 e^6\wedge(e^1\odot e^3-e^4\odot e^2).\]
The $\R$ component is 
\[\partial(\frac12w^s\otimes\omega_s)^{0,2}=w_s\hook w^{123}\otimes w_s;\]
finally, the $EH+ES^3H$ component is zero because $\chi_W$ is zero, so
\[(\Theta_B)^{0,2}=\sum_s w_s\hook w^{123}\otimes w_s.\]
The torsion of the Duchemin connection only differs in that $\Theta_D^{0,2}=0$.  
\end{example}

It is not surprising that the curvatures of the Biquard and qcm connections are related by a formula involving the torsion. The remarkable fact is that the $(2,0)$ part of this curvature only differs by a term that depends linearly on $\chi_V$:

\begin{corollary}
\label{cor:Omegabiq}
On a qcm structure, the Biquard connection has curvature
\[\Omega_B=\Omega_{qcm} + D_{qcm} T_B(\chi_V)+\frac12[T_B(\chi_V),T_B(\chi_V)].\]
In particular, 
\[(\Omega_B)^{2,0}= (\Omega_{qcm})^{2,0}+\Theta_0\hook T_B(\chi_V)\in \tilde R_1 + (S^2ES^2H +\R),\]
where the two summands intersect trivially.
\end{corollary}
\begin{proof}
The first formula is obvious. Now $[T_B(\chi_V),T_B(\chi_V)]$ has type $(0,2)$, and
 \[(D_{qcm}T_B)^{2,0} = \Theta^{2,0}\hook  T_B(\chi_V) =\Theta_0\hook T_B(\chi_V)\]
takes values in $S^2ES^2H+\R$; the relevant modules contain
\begin{gather*}
(\omega_2+i\omega_3)\otimes v_1h_2\wedge v_1h_1, \quad 
\omega_s\otimes \omega_s.
\end{gather*}
In order to prove that $S^2ES^2H$ is not contained in $\tilde R_1$,  it is sufficient to show that
\[ \delta((\omega_2+i\omega_3)\otimes v_1h_2\wedge v_1h_1) = v_jh_2\wedge v_{n+j}h_2\wedge (v_1h_2\otimes v_{1}h_1-v_{1}h_1\otimes v_1h_2)\]
does not lie in $s(V^*\otimes\widetilde{EH})$, which follows from
\begin{multline*}
s(v_1h_2\otimes v_1h_2)=\frac12 v_1h_2\wedge  (v_{n+j}h_1\wedge v_jh_2+v_{n+j}h_2\wedge v_jh_1) \otimes v_1h_2\\
 +v_1h_2\wedge v_jh_2\wedge v_{n+j}h_2  \otimes  v_1h_1
+iv_1h_2\wedge v_1h_1\wedge (w^1\otimes (w_2+iw_3)-(w^2+iw^3)\otimes w_1).
\end{multline*}
A similar computation shows that $\delta(\omega_s\otimes\omega_s)$ does not belong to the image of $s$.
\end{proof}

We can now use results from \cite{IvanovMinchevVassilev} to give a geometric characterization of the Ricci tensors.
\begin{corollary}
\label{cor:einsteiniseinstein}
On a qcm manifold, the component $\chi_V$ of the intrinsic torsion, the  Ricci tensor of the Biquard connection and the  Ricci tensor of the qcm connection coincide up to linear equivariant automorphisms of $S^2(V)$.
\end{corollary}
\begin{proof}
By Theorem~\ref{thm:qcmintrinsic} and subsequent remarks, the Ricci of the qcm connection can be identified with $\chi_V$. By Corollary~\ref{cor:Omegabiq}, there is a linear equivariant endomorphism $f$ of $S^2(EH)$ that maps $\chi_V$ to the Ricci of the Biquard connection. 

To prove that $f$ is an isomorphism, observe that by \cite{IvanovMinchevVassilev} the traceless Ricci of the Biquard connection can be identified with the component of $T_\xi$ in  $S^2ES^2H+\Lambda^2_0E$, which  can in turn be identified with  $\chi_V$ by Corollary~\ref{cor:TxiisChiV}. This shows that  $f$ is injective on $S^2ES^2H+\Lambda^2_0E$. Suppose that, for some $n$, $f$ is zero on the component $\R$ of $S^2(EH)$. This implies that on any integrable qc manifold of dimension $4n+3$, the Biquard connection has qc scalar curvature equal to zero.  The example of the sphere (see Section~\ref{sec:example} or \cite{Biquard}), which exists in all dimensions, shows that is not true.
\end{proof}


\small\noindent Dipartimento di Matematica e Applicazioni, Universit\`a di Milano Bicocca, via Cozzi 55, 20125 Milano, Italy.\\
\texttt{diego.conti@unimib.it}

\end{document}